\documentclass{article}
\usepackage[utf8]{inputenc}
\usepackage[T1]{fontenc}
\usepackage{lmodern}
\usepackage{xspace}   %%%%% commands with text

\usepackage[english]{babel}
\usepackage{verbatim}

\usepackage{amsmath}
\allowdisplaybreaks
\usepackage{amssymb}
\usepackage{amscd}
\usepackage{amsthm} %%% amsthm must be loaded after amsmath, not before.
\usepackage{amsfonts} %%% eufrak, eucal, euscript are part of amsfonts
\usepackage[normalem]{ulem} %%%%% Crossing out sentences
\usepackage{graphicx} %%% included in siam class   %%% Do not specify dvips or pdftex as an option   %%% Do not specify any extension   %%% Do not directly use psfrag.
\usepackage{tabu}
\usepackage{epic,eepic}
\usepackage{hyperref}   %%% included in siam class
\usepackage{url}
\usepackage[inline]{enumitem} %%%%% with the package option inline lists are “horizontal” 

\newtheorem{theorem}{Theorem}
\newtheorem{lemma}[theorem]{Lemma}
\newtheorem{corollary}[theorem]{Corollary}
\newtheorem{proposition}[theorem]{Proposition}
\newtheorem{prop}[theorem]{Proposition}
\newtheorem{fact}{Fact}
\newtheorem{obs}{Observation}

\newtheorem{case}{Case}

\newtheorem{remark}[theorem]{Remark}
\newcommand{\ignore}[1]{{}}
\newcommand{\defi}[1]{\emph{#1}}

\newcommand{\HC}{\operatorname{HC}}
\newcommand{\HCs}{\operatorname{HC}^*}

\newcommand{\hc}[1]{\operatorname{hc}(#1)}

\newcommand{\hcl}[1]{\operatorname{hc}_{{L}}(#1)}
\newcommand{\BG}{\operatorname{BG}}
\newcommand{\MG}{\ensuremath{M_G}}
\newcommand{\MGe}{\ensuremath{M_{G'}}}
\newcommand{\MGne}{\ensuremath{M_{G''}}}
\newcommand{\BGM}{\operatorname{BG}(M)}
\newcommand{\BMe}{\operatorname{BG}(M')}
\newcommand{\BMne}{\operatorname{BG}(M'')}
\newcommand{\BGG}{\operatorname{BG}(M_{G})}
\newcommand{\BGe}{\operatorname{BG}(M_{G'})}
\newcommand{\BGne}{\operatorname{BG}(M_{G''})}
\newcommand{\MQ}{\ensuremath{M[Q]}}
\newcommand{\BGMQ}{\operatorname{BG}(M[Q])}
\newcommand{\sfac}[1]{\operatorname{sf}(#1)}

\newcommand{\bases}{\mathcal{B}}
\newcommand{\Ba}{\ensuremath{B_1}}
\newcommand{\Bb}{\ensuremath{B_2}}
\newcommand{\Bc}{\ensuremath{B_3}}
\newcommand{\Bd}{\ensuremath{B_4}}
\newcommand{\GC}{\ensuremath{C_e}}
\newcommand{\SGC}{\ensuremath{\mathcal{C}_e}}

\newcommand{\calC}{\mathcal{C}}
\newcommand{\calP}{\mathcal{P}}
\newcommand{\north}{\ensuremath{\mathit{N}}}
\newcommand{\east}{\ensuremath{\mathit{E}}}

\title{Matroid basis graph: Counting Hamiltonian cycles
	\thanks{%Submitted to the editors August 2016.
          Research partially supported by  
          CNPq (Proc.~308523/2012-1 and~456792/2014-7), 
          FAPESP (Proc.~2012/24597-3, Proc.~2013/03447-6, and
          Proc.~2015/10323-7), 
          PICS (Grant PICS06316), and Project MaCLinC of NUMEC/USP.}
	}

\date{\empty}

\author{
	Cristina G.~Fernandes
	\thanks{Instituto de Matem\'atica e Estat\'istica, Universidade de S\~ao Paulo, 05508-090 S\~ao Paulo, Brazil.
	(\texttt{cris@ime.usp.br}, \texttt{coelho@ime.usp.br}).} 	 
\and 
	C\'esar Hern\'andez-V\'elez
	\thanks{Facultad de Ciencias, Universidad Aut\'onoma de San Luis Potos\'i, 78290 San Luis Potos\'i, Mexico. (\texttt{cesar.velez@uaslp.mx}).} 
\and 
	Jos\'e C.~de~Pina\footnotemark[2]
\and 
	Jorge~Luis Ram\'irez~Alfons\'in
	\thanks{Institut Montpelli\'erain Alexander Grothendieck, Universit\'e de Montpellier, 34095 Montpellier CEDEX~5, France. (\texttt{jorge.ramirez-alfonsin@umontpellier.fr}).}
}

\begin{document}
%\linenumbers

\maketitle

\begin{abstract}
  We present exponential and
  super factorial lower bounds on the number of Hamiltonian
  cycles passing through any edge of the basis graphs of a graphic,
  generalized Catalan and uniform matroids.
  All lower bounds were obtained by a common
  general strategy based on counting appropriated cycles of
  length four in the corresponding matroid basis graph.
\end{abstract}

% \begin{keywords}
%   matroid basis graph, generalized Catalan matroid, Hamiltonian cycle
% \end{keywords}

% \begin{AMS}
%   05B35, 05C38, 05C45
% \end{AMS}

\section{Introduction}\label{sec:intro}

For general background in matroid theory, we refer the reader to Oxley~\cite{Ox11} and Welsh~\cite{We76}.  
A \emph{matroid} ${M = (E,\bases)}$ of \emph{rank} $r=r(M)$ is a finite set $E$ together with a nonempty 
collection $\bases=\bases(M)$ of $r$-subsets of $E$, called the \emph{bases} of~$M$,
satisfying the following \emph{basis exchange axiom}:
\begin{enumerate}[label=(\textit{BEA})]
\item\label{it:bea} If $\Ba$ and $\Bb$ are members of $\bases$ and $e\in \Ba\setminus \Bb$, \\%
then there is an element $g\in \Bb\setminus \Ba$ such that $(\Ba-e)+g \in \bases$.
\end{enumerate}

%%%%%
The \emph{basis graph $\BG(M)$ of a matroid $M$} is the graph having
as vertex set the bases of $M$ and two vertices (bases) $\Ba$ and
$\Bb$ are adjacent if and only if the symmetric difference $\Ba \Delta
\Bb$ of $\Ba$ and $\Bb$ has cardinality two.  A graph is a \emph{basis
  graph} if it can be labeled to become the basis graph of some
matroid.  We make no distinction between a basis of $M$ and a vertex
of $\BG(M)$.

%%%%%%%%%%%%%%%%%%%%%%%%%%
Basis graphs have been extensively studied.
%%%%%%%%%%%%%%%%%%%%%%%%%%
Maurer~\cite{Ma73a} gave a complete characterization of
those graphs that are basis graphs.
%%%%%%%%%%%%%%%%%%%%%%%%%%%%%
Liu~\cite{Liu84,Liu88a,Liu88b} investigated the connectivity of~$\BG(M)$ and Donald, Holzmann, and Tobey~\cite{DHT77} gave a characterization of basis graphs of uniform matroids.
%%%%%%%%%%%%%%%%%%%%%%%%%%%%%
Basis graphs are closely related to \emph{matroid basis polytopes}.
%%%%%%%%%%%%%%%%%%%%%%%%%%%%%
Indeed, Gel$'$fand and Serganova~\cite{GeSe87} proved that $\BG(M)$ is the $1$-skeleton of the \emph{basis polytope} of~$M$.  
We refer the reader to the work developed by Chatelain and Ram{\'{\i}}rez~Alfons{\'{\i}}n~\cite{ChRa11,ChRa14} for further discussion and applications on this direction.

%%%%%%%%%%%%%%
A graph $G$ is \emph{edge Hamiltonian} if $G$ has order at least three and every edge is in a Hamiltonian cycle.
%%%%%%%%%%
According to Bondy and Ingleton~\cite{BoIn76}, Haff (unpublished) showed that the basis graph $\BG(M)$ of every matroid $M$ is edge
Hamiltonian, unless $\BG(M)$ is $K_1$ or $K_2$, generalizing a result due to Cummins~\cite{Cu66} and Shank~\cite{Sh68}
for graphic matroids.
So, if $\BG(M)$ has at least three vertices, then $\BG(M)$ is edge Hamiltonian. 
%%%%%%%%%%%%%%%%%%%%%%
In fact, the work of Bondy and Ingleton~\cite[Theorem~1 and Theorem~2]{BoIn76} 
about pancyclic graphs implies the edge Hamiltonicity proved by Haff.

In this paper, we investigate further the edge Hamiltonicity of~$\BG(M)$ by defining the following function.  
For a given matroid $M$, we let 
$$\HC^*(M)=\min\{\HC_e(M) : e \in E(\BG(M)) \}$$
where $\HC_e(M)$ denotes the number of different Hamiltonian cycles in $\BG(M)$ containing edge $e\in E(\BG(M))$. 
The function $\HC^*(M)$ naturally extends the edge Hamiltonicity.
Bondy and Ingleton state that $\HC^*(M)\ge 1$ for every matroid $M$.

%%%%%%%%%%%%%%%%%%%%%%%%%%%%%%%%%%%%%%%%%%%%%%%%%%%%%%%%%%%%
%%%%%%%%%% Structure of the paper
%%%%%%%%%%%%%%%%%%%%%%%%%%%%%%%%%%%%%%%%%%%%%%%%%%%%%%%%%%%%

Along this paper, when we refer that an edge $e$ is in $t$
Hamiltonian cycles, we mean that $e$ is in \emph{at least} $t$
different Hamiltonian cycles.

%%%%%%%%%%%%%%%%%%%%%%%%%%%%%%%%%%%%%%%%%%%%%%%%%%%%%%%
In Section~\ref{sec:graphicmatroids},
we give lower bounds on $\HCs(\MG)$ 
where $\MG$ is the cycle matroid
obtained from a $k$-edge-connected graph $G$.
%%%%%%%%%%%%%%%%%%%%%%%%%%%%%%%%%%%%%%%%%%%%%%
The lower bound for $k=2,3$ is exponential
on the number of vertices of $G$ (Theorems~\ref{thm:2-conn}
and~\ref{thm:hc(n,3)}).
%%%%%%%%%%%%%%%%%%%%%%%%%%%%%%%%%%%%%%%%%%%%%%
For $k \ge 4$, the lower bound is superfactorial
on $k$ and is exponetial on the number of vertices
(Theorem~\ref{thm:hc(n,k)}).
%%%%%%%%%%%%%%%%%%%%%%%%%%%%%%%%%%%%%%%%%%%%%%%%%%%%%
In Section~\ref{sec:latticepathmatroids}, we investigate
$\HC^*(M)$ when $M$ is in the class of \emph{lattice path
  matroids}.
%%%%%%%%%%%%%%%%%%%%%%%%%%%%%%%%%%%%%%%%%%%%%%%%%%%
We present a lower bound on~$\HC_e(M)$ when~$M$
is a \emph{generalized
  Catalan matroid} (Theorem~\ref{thm:latticepath}).
In particular, the derived lower bound for the
\emph{$k$-Catalan matroid} is
superfactorial on $k$.
%%%%%%%%%%%%%%%%%%%%%%%%%%%%%%%%%%%%%%%%%%%%%%%%%%%
Finally, %%%in%Section~\ref{sec:uniform},
we present a lower bound on $\HC^*(M)$ when $M$ is a
\emph{uniform matroid} (Theorem~\ref{thm:uniform}).

%%%%%%%%%%%%%%%%%%%%%%%%%%%%%%%%%%%%%%%%%%%%%%%%%%%%%%%%%%%%
%%%%%%%%%% General Strategy
%%%%%%%%%%%%%%%%%%%%%%%%%%%%%%%%%%%%%%%%%%%%%%%%%%%%%%%%%%%%

\subsection{General strategy}\label{sec:strategy}

In order to give a lower bound on $\HCs(\MG)$, we follow the strategy described below, which has the same spirit as the one used by Bondy and Ingleton~\cite{BoIn76}.

Let $M$ be a matroid and $\BGM$ be its basis graph.
Let $\Ba$ and $\Bb$ be adjacent vertices (bases) in $\BGM$.
By~\ref{it:bea}, there exist elements $e$ and $g$ of $M$, 
with ${e \in \Ba \setminus \Bb}$ and $g \in \Bb \setminus \Ba$, such that $\Bb = \Ba - e + g$.
%%%%%%%%%%%%%%%%%%%%%%%%%%%%%%%%%%%%%%%%%%%%%%%%%
We define an $(X, Y)$-bipartition (determined by $e$) of the bases of $M$, with
$X = \{ B \in \bases(M) \colon\, e \in B \}$ and ${Y = \{B \in \bases(M) \colon\, e \not\in B \}}$. 
The bases in $X$ ($Y$, respectively) correspond exactly to the bases of the matroid $M' = M / e$ obtained by \defi{contracting} $e$ ($M'' = M \setminus e$, obtained by \defi{deleting} $e$, respectively).
Moreover, $\BMe$ ($\BMne$, respectively) is $\BGM[X]$ ($\BGM[Y]$, respectively), which is the subgraph of $\BGM$ induced by $X$ ($Y$, respectively).
Therefore, there is a $1{-}1$ correspondence between  Hamiltonian cycles of $\BMe$ ($\BMne$, respectively) and  Hamiltonian cycles of $\BGM[X]$ ($\BGM[Y]$, respectively).
For readability, we do not distinguish between $\BMe$ ($\BMne$, respectively) and $\BGM[X]$ ($\BGM[Y]$, respectively).
%%%%%%%%%%%%%%%%%%%%%%%%%%%%%%%%%%%%%%%%%%%%%%%%%

A basis sequence $\Ba\Bb\Bc\Bd$ is a \emph{good cycle} for $\Ba\Bb$ if
it is a cycle (of length four) in $\BGM$, each of $\Ba$ and $\Bd$
contains $e$, and none of~$\Bb$ and~$\Bc$ contains $e$; that is, $\Ba$
and~$\Bd$ are adjacent bases of~$\BMe$ and $\Bb$ and~$\Bc$ are
adjacent bases of~$\BMne$ (Figure~\ref{fig:key}).

\begin{figure}[h]
	\centering
	\includegraphics[width=0.4\linewidth]{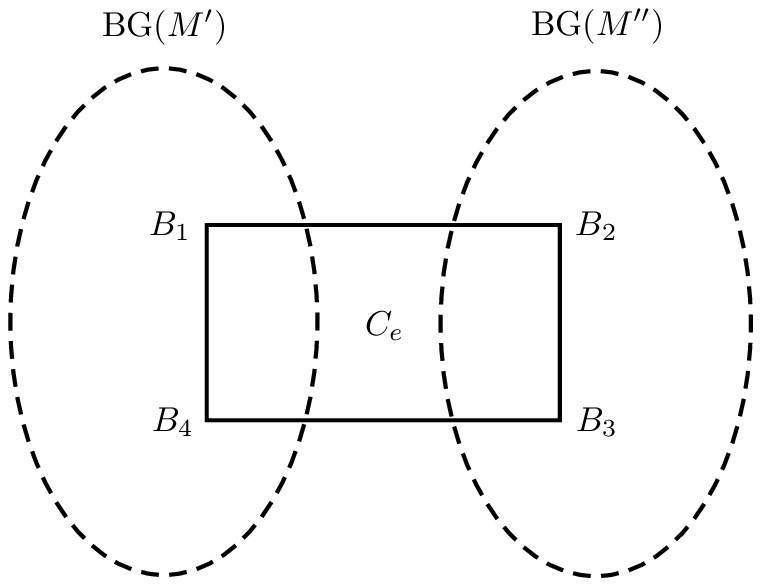}
	\caption{A good cycle $\GC=\Ba\Bb\Bc\Bd$ for $\Ba\Bb$.}
	\label{fig:key}
\end{figure}

If $\GC=\Ba\Bb\Bc\Bd$ is good, then the symmetric difference of
a Hamiltonian cycle of $\BMe$ passing through the edge $\Ba\Bd$,  
the good $\GC$, and a Hamiltonian cycle of~$\BMne$ passing through
the edge $\Bb\Bc$ is a Hamiltonian cycle of $\BGM$.

So, if $\calC(\Ba,\Bb)$ is the set of good cycles for $\Ba\Bb$, then 
\[\HC_{\Ba\Bb}(M) \ge \HCs(M') \times |\calC(\Ba,\Bb)| \times \HCs(M'').\]
This inequality suggests an inductive way to achieve a lower bound on $\HCs(M)$.  
A key part in this approach involves proving a lower bound on the number of good cycles for any edge of $\BGM$.

%%%%%%%%%%%%%%%%%%%%%%%%%%%%%
%%%%%%%%% Section
%%%%%%%%%%%%%%%%%%%%%%%%%%%%%

\section{Graphic matroids}\label{sec:graphicmatroids}

In this section, we consider a graphic matroid $M_G$ where~$G$ is a
$k$-edge-connected graph of order $n$; that is, the elements of the
ground set of~$M_G$ are the edges of~$G$ and a basis of~$M_G$
corresponds to a spanning tree of~$G$, thus a basis of~$M_G$ contains
exactly~$n-1$ edges of~$G$.  Since loops of~$G$ are in no basis
of~$M_G$, we always consider graphs with no loops.
For readability, we
do not distinguish between a basis of~$M_G$ and a spanning tree
of~$G$.
If~$B$ is a basis of~$M_G$ and~$g$ is an edge of~$G$ not in~$B$,
then~$B+g$ induces a unique cycle (circuit) $C(g,B)$ in~$G$ (in $M_G$,
respectively) called the \emph{fundamental cycle}
(\emph{circuit}, respectively) with respect to~$g$
and $B$~\cite{Ox11}.

First, note that, by Haff's result,
if~$G$ is a $k$-edge-connected graph of order
$n \ge 3$, for $k \ge 2$, then the graph~$\BG(M_G)$
has at least three vertices and is edge Hamiltonian. 

%%%%%%%%%%%%%%%%%%%%%%%%%%%%%%%%%%%%%%%%%%%%%%%%%
Let~$G' = G/e$ be the graph resulting from contracting the edge $e$ of $G$ and then removing loops and 
let $G'' = G\setminus e$ be the graph resulting from deleting the edge $e$.
%%%%%%%%%%%%%%%%%%%%%%%%%%%%%%%%%%%%%%%%%%%%%%%%%

%%%%%%%%%%%%%%%%%%%%%%%%%%%%%%%%%%%%%%%%%%%%%%%%%

Let $X$ and $Y$ be disjoint subsets of the vertex set $V(G)$.  We
denote by $E[X,Y]$ ($= E[Y,X]$) the set of edges of $G$ with one end
in $X$ and the other end in $Y$, and by~$e(X,Y)$ their number.

%%%%%%%%%%%%%%%%%%%%%%%%%%%%%%%%%%%%%%%%%%%%%%%%%%%%%%%%%%%%%%%%%%%%%%%%
%%%%%%%%%%%%%%%%%%%%  Good cycles
%%%%%%%%%%%%%%%%%%%%%%%%%%%%%%%%%%%%%%%%%%%%%%%%%%%%%%%%%%%%%%%%%%%%%%%%

\subsection{General structure of good cycles}\label{sec:goodcycles}

Now, we fix the structure that we will use in the rest of Section~\ref{sec:graphicmatroids} and, 
unless otherwise stated, we will follow this notation.
%%%%%%%%%%%
The facts presented ahead show types of good cycles that this structure induces.

%%%%%%%%%%%%%%
Let~$G$ be a graph and~$\Ba$ and $\Bb$ be bases of~$M_G$ such that $\Bb = \Ba -e +g$.
Let $f$ be an edge of~$\Ba-e$.
Let~$X$ be the vertex set of the component of $\Ba-e$ that contains no end of~$f$.
Let $Z$ be the vertex set of the component of $\Ba-f$ that contains no end of~$e$.
Let $Y = V(G) \setminus (X \cup Z)$.

Let $\SGC = \calC(\Ba,\Bb)$ be the set of good cycles for $\Ba\Bb$.  
An arbitrary element of~$\SGC$ is denoted by $\GC$, and is represented as $\Ba\Bb\Bc\Bd$. 
For ${f \in \Ba-e = \Bb-g}$, let~$\SGC(f) = \{\GC \in \SGC : f \not\in \Bd\}$. 
An arbitrary element of $\SGC(f)$ is denoted by $\GC(f)$. 
For every $f' \in \Ba - e$ with $f'\neq f$, since~$f'$ belongs to both $\Bc$ and~$\Bd$ for every cycle~$\GC(f)$, 
we have that ${\SGC(f) \cap \SGC(f') = \emptyset}$.  Thus $\SGC = \dot{\bigcup} \{\SGC(f) : f \in \Ba - e\}$.
For every $w \not\in \Ba + g = \Bb + e$, we denote by $\SGC(f,w)$ the set of cycles in $\SGC(f)$ such that~$w \in \Bc$. 
Similarly, ${\SGC(f,w) \cap \SGC(f,w') = \emptyset}$ for every $w' \not\in \Ba + g$ with $w' \neq w$.
Therefore ${\SGC(f) = \dot{\bigcup} \{\SGC(f,w) : w \not\in \Ba + g\}}$. 
Summarizing, the following holds. 
 
\begin{remark}\label{rem:disjointness}
$\SGC(f) \cap \SGC(f') = \emptyset$ and $\SGC(f,w) \cap \SGC(f,w') = \emptyset$ 
for every $f, f' \in \Ba - e$ with $f \neq f'$ and every $w, w' \not\in \Ba + g$ with $w \neq w'$. 
\end{remark}

%%%%%%%%%%%%%%%%%%%%%%%%%%%%%%
%%%%%%%%%% Fact
%%%%%%%%%%%%%%%%%%%%%%%%%%%%%%
\begin{fact}\label{fact:XY-Z}
	If $f$ is not in $C(g,\Ba)$ and $w$ is an edge in $E[X\cup Y,Z]$ other than~$f$, 
	then there exists a good cycle $\GC(f,w)$ by defining
	\begin{itemize}
		\item $\Bd = \Ba -f+w$ and $\Bc = \Bb -f+w$.
	\end{itemize}
	Note that $\Bc = \Bd -e+g$. (See Figure~\ref{fig:EdgeXY-Z}.)
\end{fact}

\begin{figure}[h]
	\centering
	\begin{minipage}[c]{0.5\linewidth}
		\centering
		\includegraphics[width=0.5\linewidth]{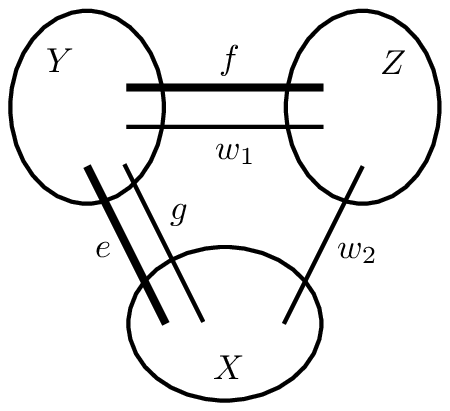}
		%\caption{Figure…}
		%\label{fig:}
	\end{minipage}%
	\begin{minipage}[c]{0.5\linewidth}
		\centering
		{\footnotesize 
			\begin{tabular}{|c|l l| c|}\hline
				$\Ba$ & $\{e,f,\ldots\}$ & $\{g,f,\ldots\}$ & $\Bb$\\
				$\Bd$ & $\{e,w,\ldots\}$ & $\{g,w,\ldots\}$ & $\Bc$\\
				\hline
			\end{tabular}
		}
		%\captionof{table}{Table…}
		%\label{tab:…}
	\end{minipage}
	\caption{Edge $f$ is in $\Ba$. There are edges~$w_1$ and $w_2$ between $X\cup Y$ and $Z$. 
		The table shows a good $\GC(f,w)$'s containing $\Ba\Bb$.}
	\label{fig:EdgeXY-Z}
\end{figure}

%%%%%%%%%%%%%%%%%%%%%%%%%%%%%%
%%%%%%%%%% Fact
%%%%%%%%%%%%%%%%%%%%%%%%%%%%%%
\begin{fact}\label{fact:YZ}
	If $f$ is in $C(g,\Ba)$ and $\ell$ is an edge in $E[Y,Z]$ other than $f$, 
	then there are two good cycles $\GC(f,\ell)$ by defining
	\begin{itemize}
		\item $\Bd = \Ba -f+\ell$ and $\Bc = \Bb -f+\ell$.
		\item $\Bd = \Ba -f+g$    and $\Bc = \Bb -f+\ell$.
	\end{itemize}
	Note that, in the first case, $\Bc = \Bd -e+g$ and, in the second case, $\Bc = \Bd -e+\ell$. (See Figure~\ref{fig:EdgeYZ}.)
\end{fact}

\begin{figure}[h]
	\centering
	\begin{minipage}[c]{0.5\linewidth}
		\centering
		\includegraphics[width=0.5\linewidth]{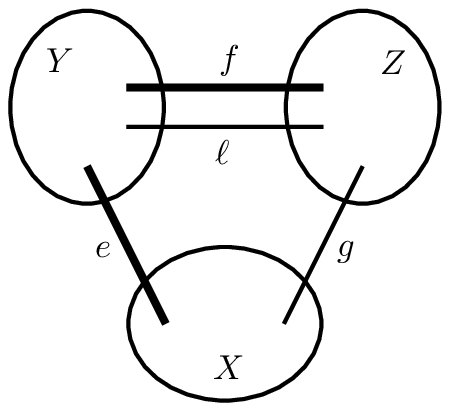}
		%\caption{Figure…}
		%\label{fig:…}
	\end{minipage}%
	\begin{minipage}[c]{0.5\linewidth}
		\centering
		{\footnotesize
			\begin{tabular}{|c|l l| c|}\hline
				$\Ba$ & $\{e,f,\ldots\}$ & $\{g,f,\ldots\}$ & $\Bb$\\
				$\Bd$ & $\{e,\ell, \ldots\}$ & $\{g,\ell, \ldots\}$ & $\Bc$\\
				\hline
				$\Ba$ & $\{e,f,\ldots\}$ & $\{g,f,\ldots\}$ & $\Bb$ \\
				$\Bd$ & $\{e,g,\ldots\}$ & $\{g,\ell, \ldots\}$ & $\Bc$\\
				\hline
			\end{tabular}
		}
		%\captionof{table}{Table…}
		%\label{tab:…}
	\end{minipage}
	\caption{The bold edges are in $\Ba$. There is an edge $\ell$ between $Y$ and $Z$. The table shows two good cycles~$\GC(f,\ell)$ containing $\Ba\Bb$.}
	\label{fig:EdgeYZ}
\end{figure}

%%%%%%%%%%%%%%%%%%%%%%%%%%%%%%
%%%%%%%%%% Fact
%%%%%%%%%%%%%%%%%%%%%%%%%%%%%%
\begin{fact}\label{fact:XY}
	If $f$ is in $C(g, \Ba)$ and $h$ is an edge in $E[X,Y]$ other than $e$,
	then there exists a good cycle $\GC(f,h)$ by defining
	\begin{itemize}
		\item $\Bd = \Ba -f+g$ and $\Bc = \Bb -f+h$.
	\end{itemize}
	Note that $\Bc = \Bd -e+h$. (See Figure~\ref{fig:EdgeX-YZ}.)
\end{fact}

%%%%%%%%%%%%%%%%%%%%%%%%%%%%%%
%%%%%%%%%% Fact
%%%%%%%%%%%%%%%%%%%%%%%%%%%%%%
\begin{fact}\label{fact:XZ}
	If $f$ is in $C(g,\Ba)$ and $j$ is an edge in $E[X,Z]$ other than $g$,
	then there exists a good cycle $\GC(f,j)$ by defining
	\begin{itemize}
		\item $\Bd = \Ba -f+j$ and $\Bc = \Bb -g+j$.
	\end{itemize}
	Note that $\Bc = \Bd -e+f$. (See Figure~\ref{fig:EdgeX-YZ}.)
\end{fact}

\begin{figure}[h]
	\centering
	\begin{minipage}[c]{0.5\linewidth}
		\centering
		\includegraphics[width=0.5\linewidth]{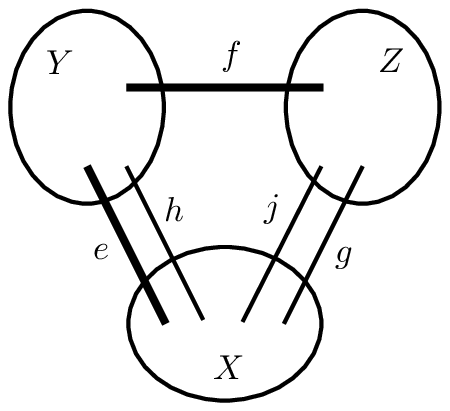}
		%\caption{Fiure…}
		%label{fig:…}
	\end{minipage}%
	\begin{minipage}[c]{0.5\linewidth}
		\centering
		{\footnotesize
			\begin{tabular}{|c|l l|c|}\hline
				$\Ba$ & $\{e,f,\ldots\}$ & $\{g,f,\ldots\}$ & $\Bb$\\
				$\Bd$ & $\{e,g,\ldots\}$ & $\{g,h,\ldots\}$ & $\Bc$\\
				\hline
				$\Ba$ & $\{e,f,\ldots\}$ & $\{g,f,\ldots\}$ & $\Bb$\\
				$\Bd$ & $\{e,j,\ldots\}$ & $\{j,f,\ldots\}$ & $\Bc$\\
				\hline
			\end{tabular}
		}
		%\captionof{table}{Table…}
		%\label{tab:…}
	\end{minipage}
	\caption{The bold edges are in $\Ba$. There is an edge $h$ between $X$ and $Y$, and an edge $j$ between~$X$ and $Z$. The table shows the two good cycles $\GC$ containing $\Ba\Bb$.}
	\label{fig:EdgeX-YZ}
\end{figure}

%%%%%%%%%%%%%%%%%%%%%%%%%%%%%%%%%%%%%%%%%%%%%%
%%%%%%%%%%%%%%%%%%%%%%%%%%%%%%%%%%%%%%%%%%%%%%
%\newpage

%%%%%%%%%%%%%%%%%%%%%%%%%%%%%%%%%%%%%%%%%%%%%%%%%%
%%%%%%%%%%
%%%%%%%%%%%%%%%%%%%%%%%%%%%%%%%%%%%%%%%%%%%%%%%%%%
\subsection{\texorpdfstring{$2$}{2}-edge-connected graphs}

We start by giving a lower bound on $\HC^*(M_G)$ where~$M_G$ is the
cycle matroid obtained from a $2$-edge-connected graph $G$.  In what
follows we shall use the notation introduced in the beginning of this section.
%%%%%%%%%%%%%%%%%%%%%%%%%%%%%%%%%%%%%%%%%%%%%%%%%%%%%%%%%%%%%
In particular, we use extensively the facts and the structure
of the vertex sets $X$, $Y$, $Z$ provided by adjacent bases $\Ba$, $\Bb$ and
a convenient edge $f$ in $\Ba \cap \Bb$.

%%%%%%%%%%%%%%%%%%%%%%%%%%%%%%
%%%%%%%%%% Prop: aresta amarela
%%%%%%%%%%%%%%%%%%%%%%%%%%%%%%
\begin{prop}\label{prop:C4-ye-tail}
  Let $G$ be a $2$-edge-connected graph.
  %%%%%%%%%%%%%%%%%%%%%%%%%%%%%%%%%%
  Let $\Ba$ and $\Bb$ be adjacent bases of $\BG(M_G)$,
  say $\Bb = \Ba-e+g$.
  %%%%%%%%%%%%%%%%%%%%%%%%%%%%%%%%%%%
  Each edge $f$ of $\Ba$ with at most one end in~$C(g,\Ba)$ provides a good cycle $\GC(f)$.
\end{prop}

\begin{proof}
  It follows from Fact~\ref{fact:XY-Z} that, given an edge $f$ in $\Ba$ with at most one end in~$C(g,\Ba)$, 
  for every edge $w \in E[X \cup Y, Z]$ (one such $w$ exists because $G$ is $2$-edge-connected), there exists a good cycle $\GC(f,w)$.
\end{proof}

%%%%%%%%%%%%%%%%%%%%%%%%%%%%%%
%%%%%%%%%% Prop: aresta branca
%%%%%%%%%%%%%%%%%%%%%%%%%%%%%%
\begin{prop}\label{prop:C4-we-tail}
  Let $G$ be a connected graph.
  %%%%%%%%%%%%%%%%%%%%%%%%%%%%%%%%%%
  Let $\Ba$ and $\Bb$ be adjacent bases of~$\BG(M_G)$, say $\Bb = \Ba-e+g$.
  %%%%%%%%%%%%%%%%%%%%%%%%%%%%%%%%%%%
  For each edge $w$ not in~$\Ba$ with at most one end in~$C(g,\Ba)$, there exists an edge $f_w \in \Ba - e$ that provides a good cycle~$\GC(f_w,w)$.
\end{prop}

\begin{proof}
  Let $w$ be an edge not in $\Ba$ with at most one end in~$C(g,\Ba)$.
  As at most one end of $w$ is in~$C(g,\Ba)$, there exists an edge $f_w$ of $\Ba-C(g,\Ba) \subseteq \Ba - e$ in the fundamental cycle~$C(w,\Ba)$.
  It follows from Fact~\ref{fact:XY-Z} that there exists a good cycle~$\GC(f_w,w)$.
\end{proof}

%%%%%%%%%%%%%%%%%%%%%%%%%%%%%%
%%%%%%%%%% Prop: branca com pontas no circuito
%%%%%%%%%%%%%%%%%%%%%%%%%%%%%%
\begin{prop}\label{prop:C4-we-chord}
  Let $G$ be 2-edge-connected graph.
  Let $\Ba$ and $\Bb$ be adjacent bases of $\BG(M_G)$, say $\Bb = \Ba-e+g$.
  Suppose that $C(g,\Ba)$ has length at least three.
  For each edge $w$ not in $\Ba + g$ with both ends in $C(g,\Ba)$, there exists an edge $f_w \in \Ba - e$ that provides a good cycle $\GC(f_w,w)$. 
  %%%%%%%%%%%%%%%%%%%%%%%%%%%%%%%%%%%%%%%%%%%
  %% For every edge not in $\Ba+g$ and with both ends in
  %% $C(g,\Ba)$ there exists a good cycle of
  %% length four containing $\Ba\Bb$.
\end{prop}

\begin{proof}
  As $C(g,\Ba)$ has length at least three, $e$ and $g$ are not parallel edges.
  Let $w$ be an edge not in $\Ba + g$ with both ends in $C(g,\Ba)$.
    
  \begin{case}
    The edge $w$ is parallel to $g$.
  \end{case}
  
  Let $f_w$ be an edge of $C(g,\Ba)-e-g$.
  In this case $w$ is as $j$ in Fact~\ref{fact:XZ}.
  
  \begin{case}
    The edge $w$ is not parallel to $g$ and the fundamental cycle $C(w,\Ba)$ contains the edge $e$.
  \end{case}
  
  Let $f_w$ be an edge of $C(g,\Ba)-e-g$ and not in $C(w,\Ba)$.
  In this case $w$ is as $h$ in Fact~\ref{fact:XY}.
  
  \begin{case}
    The edge $w$ is not parallel to $g$ and the fundamental cycle $C(w,\Ba)$ does not contain the edge $e$.
  \end{case}
  
  Let $f_w$ be an edge of $C(w,\Ba)-w \subseteq \Ba - e$. 
  In this case $w$ is as $\ell$ in Fact~\ref{fact:YZ}.

So, each case leads to one of the previously stated facts where we obtain an $f_w$ and a good cycle $\GC(f_w,w)$. 
\end{proof}

%%%%%%%%%%%%%%%%%%%%%%%%%%%%%%
%%%%%%%%%% Lemma
%%%%%%%%%%%%%%%%%%%%%%%%%%%%%%
\begin{lemma}\label{lem:C4-2conn}
  If $G$ is a $2$-edge-connected graph of order $n \ge 4$ and size at least $n+2$, 
  then every edge of $\BG(M_G)$ is in two good cycles.
\end{lemma}

\begin{proof}
  Let $\Ba$ and $\Bb$ be adjacent bases of $\BG(M_G)$, say $\Bb = \Ba -e +g$.
  
  Suppose that $e$ and $g$ are parallel edges; that is, $C(g,\Ba)$ is the $2$-cycle $eg$.
  Since~$G$ has order $n \ge 4$, there are two edges in $\Ba - e$.
  By Proposition~\ref{prop:C4-ye-tail}, each one of them gives a good cycle, 
  and they are distinct by Remark~\ref{rem:disjointness}. 
  
  Now, suppose that $e$ and $g$ are not parallel edges in $G$.
  Thus, $C(g,\Ba)$ has length at least three.
  If there are two edges in $\Ba$ with at most one end in~$C(g,\Ba)$ or two edges not in~$\Ba$ 
  with at most one end in~$C(g,\Ba)$, by Propositions~\ref{prop:C4-ye-tail} and~\ref{prop:C4-we-tail}, respectively,  
  we have two good cycles, distinct by Remark~\ref{rem:disjointness}, so we are done.
  %%%%%
  Also, if there are two edges not in~$\Ba+g$ with both ends in~$C(g,\Ba)$, 
  then we are done by Proposition~\ref{prop:C4-we-chord} and Remark~\ref{rem:disjointness}. 
  
  Finally, as $G$ has size at least $n+2$, we may assume there exist an edge in~$\Ba$ with at most one end in~$C(g,\Ba)$ 
  and an edge not in $\Ba+g$ with both ends in~$C(g,\Ba)$. 
  Therefore, by Propositions~\ref{prop:C4-ye-tail} and~\ref{prop:C4-we-chord}, respectively, 
  and Remark~\ref{rem:disjointness}, the lemma follows.
\end{proof}

The \defi{$1$-sum} $H \oplus_1 H'$ of two graphs $H$ and $H'$ is the graph obtained from identifying a vertex of $H$ with a vertex of $H'$.

%%%%%%%%%%%%%%%%%%%%%%%%%%%%%%
%%%%%%%%%% Lemma
%%%%%%%%%%%%%%%%%%%%%%%%%%%%%%
\begin{lemma}\label{lem:no2good}
Let $G$ be a $2$-edge-connected graph of order $n \ge 4$.  There exists an
edge in $\BG(M_G)$ not in two good cycles if and only if $G$ is
either~$C_n$ or~$C_2 \oplus_1 C_{n-1}$.
\end{lemma}

\begin{proof}
Let $m$ denote the number of edges of $G$.
Since $G$ is $2$-edge-connected, every edge is in a cycle, so $m \ge n$.
If $m = n$, then $G$ is the $n$-cycle $C_n$ and no edge of $\BG(M_G)$ is in a good cycle.
For $m \ge n+2$, Lemma~\ref{lem:C4-2conn} implies that every edge of $\BG(M_G)$ is in two good cycles.
So, we may assume that $m = n+1$.
Because every $2$-edge-connected graph has a closed ear-decomposition~\cite{BoMu08}, and $G$ has exactly $m+1$ edges, the closed ear-decomposition of $G$ consist of exactly two ears.
Thus, $G$ is either
\begin{enumerate}[label=\roman*)]
	\item\label{it:peixe} The 1-sum of two cycles, or
	\item\label{it:3paths} The union of three internally disjoint paths that have the same two end vertices.
\end{enumerate}

First, suppose that $G$ is the $1$-sum of two cycle, say $C^1 \oplus_1 C^2$.
Since we only consider graphs with no loops, the length of both $C^1$ and $C^2$ is at least two.
When the length of both $C^1$ and $C^2$ is at least three, Proposition~\ref{prop:C4-ye-tail} provides two good cycles for every edge of $\BG(M_G)$.
Therefore, $G$ is $C_2 \oplus_1 C_{n-1}$ and it can be verified that there are adjacent bases $\Ba$ and $\Bb$ in $\BGG$ 
for which there is only one good cycle (Figure~\ref{fig:notwogoods}).

Now, suppose that $G$ is the union of three internally disjoint paths, say $P_1$, $P_2$,~$P_3$, that have the same two end vertices.
In this case we shall show that every edge of~$\BG(M_G)$ is in two good cycles.

Let $\Ba\Bb$ be an edge of $\BG(M_G)$, say $\Bb= \Ba-e+g$.
First, suppose that $e$ and~$g$ are in the same path, say $P_1$.
Thus, without loss of generality, all edges of $P_2$ are in~$\Ba$, and there exists an edge $w$ in $P_3$ not in $\Ba$.
Let $f$ be an edge of $P_2$ (and therefore of~$\Ba$).
Keeping our notation defined in Section~\ref{sec:goodcycles}, $w$ is in $E[Y,Z]$ and Fact~\ref{fact:YZ} provides two good cycles $\GC(f,w)$.

Finally, suppose that $e$ and $g$ are in different paths;
say $e$ belongs to $P_1$ and $g$ belongs to $P_2$.
Thus, all edges of $P_1$ are in $\Ba$, and there exists an edge $w$ in $P_3$ not in~$\Ba$.
If there exists an edge $f \in \Ba-e$ in $P_1$, then $w$ is in $E[X,Z]$ as $j$ in Fact~\ref{fact:XZ}. 
If there exists an edge $f \in \Ba$ in $P_2$, then $w$ is in $E[X,Y]$ as $h$ in Fact~\ref{fact:XY}.
If there exists an edge $f \in \Ba$ in $P_3$, then $w$ is in $E[X\cup Y,Z]$ as in Fact~\ref{fact:XY-Z}.
In any case we get a good cycle $\GC(f,w)$.
Since $G$ has order at least four, there are two edges $f, f' \in\Ba$ other than $e$.
Therefore, by Remark~\ref{rem:disjointness}, every edge $\Ba\Bb$ is in two good cycles, named $\GC(f,w)$ and $\GC(f',w)$.
\end{proof}

\begin{figure}[h]
	\centering	
	\begin{minipage}{0.7\linewidth}
	\centering
	\includegraphics[width=0.25\linewidth]{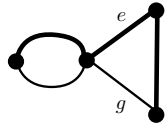}
	\caption{A $2$-edge-connected graph $G$ whose basis graph $\BGG$ has an edge, $\Ba\Bb$, with no two good cycles. The basis $\Ba$ is the spanning tree in thick edges and $\Bb = \Ba - e + g$.}
	\label{fig:notwogoods}
	\end{minipage}
\end{figure}

%%%%%%%%%%%%%%%%%%%%%%%%%%%%%%
%%%%%%%%%% Proposition
%%%%%%%%%%%%%%%%%%%%%%%%%%%%%%
\begin{prop}\label{prop:2conn-m=n}
For $n \ge 3$, 
every edge of~$K_n$ is in~$(n-2)!$ Hamiltonian cycles.
\end{prop}

\begin{proof}
Let $\{v_1,v_2, \ldots, v_n\}$ be the vertex set of $K_n$.
Consider the edge $v_iv_j$. 
Without loss of generality, we may assume that $v_iv_j = v_1v_2$.
Note that $v_1v_2v_{\sigma(3)}\ldots v_{\sigma(n)}$ is a Hamiltonian cycle for any permutation $\sigma$ of $\{3,\ldots, n\}$.
Therefore, the number of Hamiltonian cycles containing the edge $v_1v_2$ is $(n-2)!$ and the lemma follows.
\end{proof}

%%%%%%%%%%%%%%%%%%%%%%%%%%%%%%
%%%%%%%%%% Proposition
%%%%%%%%%%%%%%%%%%%%%%%%%%%%%%
\begin{prop}\label{prop:cycle+}
  For $n \ge 3$, every edge of $K_2 \square K_{n-1}$
  is in $(n-2)!(n-3)!$ Hamiltonian cycles.
%  For $n \ge 2$, every edge of $K_2 \square K_n$
%  is in $(n-1)!(n-2)!$ Hamiltonian cycles.
\end{prop}

\begin{proof}
  If $n = 3$, then $(n-2)!(n-3)! = 1$ and $K_2 \square K_2$ is $C_4$.
  So, we may assume that $n \ge 4$.

Let $\{u_1, u_2\}$ be the vertex set of $K_2$, and $\{v_1, \ldots,
v_{n-1}\}$ be the vertex set of~$K_{n-1}$.
%%%%%%%%%%%%%%%%%%%%%%%%%%%%%%%%%%%%%%%%%%%%%%%%%%%%%%%%
Let ${K_{n-1}^1 =  \bigl\{(u_1, v_1), \ldots, (u_1, v_{n-1})\bigr\}}$ and
$K_{n-1}^2 = \{(u_2,v_1), \ldots, (u_2, v_{n-1})\bigr\}$.
There exists a natural partition
$(K_{n-1}^1,K_{n-1}^2)$ of the vertex set of $K_2 \square K_{n-1}$. 
%%%%%%%%%%%%%%%%%%%%%%%%%%%%%%%%%%%%%%%%%%%%%%%%%%%%%%%%
For every permutation~$\sigma$ of $\{1, \ldots, n-1\}$,
let $C_{\sigma}^\ell$ be the Hamiltonian cycle
$(u_\ell,v_{\sigma(1)})\cdots(u_\ell,v_{\sigma(n-1)})$ of~$K_{n-1}^\ell$.
Consider an edge $(u_{\ell}, v_i)(u_{\ell},v_j)$ of $K_2 \square K_{n-1}$. 
 Without loss of generality we may assume that $\ell=i=1$ and $j=2$.
For a permutation $\sigma_1$ with $\sigma_1(1)=1$ and $\sigma_1(2)=2$,
let $(u_1,v_x)(u_1,v_y)$ be an edge of $C_{\sigma_1}^1$ other than
$(u_1,v_1)(u_1,v_2)$.
For every permutation~$\sigma_2$ such that $\sigma_2(1)=x$ and $\sigma_2(2)=y$, 
the Hamiltonian cycle $C_{\sigma_2}^2$ of $K_{n-1}^2$ uses the edge $(u_2,v_x)(u_2,v_y)$.
Let $S$ denote the cycle $(u_1,v_x)(u_2,v_x)(u_2,v_y)(u_1,v_y)$.
The symmetric difference $C_{\sigma_1}^1 \Delta S \Delta C_{\sigma_2}^2$ is a Hamiltonian cycle of $\BG(M_G)$ containing the edge $(u_1,v_1)(u_1,v_2)$.
Since the edge $(u_1,v_x)(u_1,v_y)$ can be chosen in $n-2$ different ways, and the number of permutations~$\sigma_1$ as well as the number of permutations~$\sigma_2$ is~$(n-3)!$, 
we obtain $(n-2)\bigl((n-3)!\bigr)^2$ Hamiltonian cycles passing through the edge~$(u_1,v_1)(u_1,v_2)$.

Now consider an edge $(u_1,v_i)(u_2,v_i)$ of $K_2 \square K_{n-1}$.
Without loss of generality we may assume  that $i=1$.
Consider the edge $(u_1,v_1)(u_1,v_j)$.
Such edge is in ${(n-3)!}$ Hamiltonian cycles $C^1$ of $K^1_{n-1}$.
On the other hand, there are $(n-3)!$ Hamiltonian cycles $C^2$ of $K^2_{n-1}$
passing through the edge $(u_2,v_1)(u_2,v_j)$.
Let $S$ denote the cycle $(u_1,v_1)(u_1,v_j)(u_2,v_j)(u_2,v_1)$.
The symmetric difference $C^1 \Delta S \Delta C^2$ is a Hamiltonian cycle of $\BG(M_G)$ containing the edge $(u_1,v_1)(u_2,v_1)$.
Since there are $(n-3)!$ cycles $C^1$, as well as cycles $C^2$,
and the edge $(u_1,v_1)(u_1,v_j)$ can be chosen in $n-2$ different ways, 
we obtain $(n-2)\bigl((n-3)!\bigr)^2$ Hamiltonian cycles containing the edge $(u_1,v_1)(u_2,v_1)$. 
Note that $(n-2)\bigl((n-3)!\bigr)^2=(n-2)!(n-3)!$, hence the proposition follows.
\end{proof}

\begin{theorem}\label{thm:2-conn}
  If $G$ is a $2$-edge-connected graph of order $n \ge 3$, 
  then every edge of~$\BG(M_G)$ is in $2^{n-3}$ Hamiltonian cycles. 
%\endproof 
\end{theorem}

\begin{proof}
  %%%%%%%%%%%%%%%%%%%%%%%%%%%%%%%%%%%%%%%%%%%%%%%%%
  The proof is by induction on $n$.
  If $n = 3$, then $2^{n-3} = 1$ and the theorem follows from the edge Hamiltonicity of $\BG(M_G)$.
  So we may assume that $n \geq 4$.

  %%%%%%%%%%%%%%%%%%%%%%%%%%%%%%%%%%%%%%%%%%%%%%%%%
  By Lemma~\ref{lem:no2good}, if there exists an edge in $\BG(M_G)$ not in two good cycles, 
  then $G$ is either $C_n$ or $C_2 \oplus_1 C_{n-1}$.
  %%%%%%%%%%%%%%%%%%%%%%%%%%%%%%%%%%%%%%%%%%%%%%%%%%%%%%%%%%%%%%%%%%%%%
  If $G=C_n$, then  $\BG(M_G) = K_n$ and by 
  Proposition~\ref{prop:2conn-m=n} every edge of~$\BG(M_G)$ is in
  $(n-2)! \geq 2^{n-3}$ Hamiltonian cycles.
  %%%%%%%%%%%%%%%%%%%%%%%%%%%%%%%%%%%%%%%%%%%%%%%%%%%%%%%%%%%%%%%%%%%%%
  If $G=C_2 \oplus_1 C_{n-1}$, then
  $\BG(M_G)=K_2 \square K_{n-1}$ and 
  by Proposition~\ref{prop:cycle+}  
  every edge of~$\BG(M_G)$ is in ${(n-2)!(n-3)!} \geq 2^{n-3}$
  Hamiltonian cycles. 
  Therefore, we may assume that~$G$ has at least $n+1$ edges and 
  every edge of $\BG(M_G)$ is in two good cycles.

  %%%%%%%%%%%%%%%%%%%%%%%%%%%%%%%%%
  Let $\Ba\Bb$ be an edge of $\BG(M_G)$, say $\Bb = \Ba-e+g$.
  %%%%
  Let $G' = G / e$ and~${G'' = G \setminus e}$.
  As $G'$ is $2$-edge-connected of order $n-1 \ge 3$, by the induction hypothesis, 
  every edge of $\BGe$ is in~$2^{n-4}$ Hamiltonian cycles in~$\BGe$.
  %%%%
  As~$G''$ has $n \ge 4$ vertices and at least $n$ edges, $G''$ has at least two spanning trees, 
  and therefore $\BGne$ is either $K_2$ or edge Hamiltonian.

  Let $\GC=\Ba\Bb\Bc\Bd$ be a good cycle. 
  %%%%%%%%%%%%%%%%%%%%%%%%%%%%%%%%%%%%%%%%%%%%%%%%%%%%%%%%
  If~$\BGne$ is $K_2$, then
  the symmetric difference of $\GC$ and a Hamiltonian cycle of $\BGe$
  containing $\Ba\Bd$ is a Hamiltonian cycle of $\BG(M_G)$.
  On the other hand, if $\BGne$ is edge Hamiltonian, then
  the symmetric difference of $\GC$, a Hamiltonian cycle of~$\BGe$
  containing $\Ba\Bd$, and a Hamiltonian cycle of~$\BGne$ containing $\Bb\Bc$ 
  is a Hamiltonian cycle of $\BG(M_G)$.
  %%%%%%%%%%%%%%%%%%%%%%%%%%%%%%%%%%%%%%%%%%%%%%%%%%%%%%%%%
  As every edge of $\BG(M_G)$ is in two good cycles, in either case
  we conclude that every edge of $\BG(M_G)$ is in~$2^{n-4}\cdot 2 \cdot 1 = 2^{n-3}$ Hamiltonian cycles.

\end{proof}

%\newpage

%%%%%%%%%%%%%%%%%%%%%%%%%%%%%%%%%%%%%%%%%%%%%%%%%%
%%%%%%%%%% k-edge-connected
%%%%%%%%%%%%%%%%%%%%%%%%%%%%%%%%%%%%%%%%%%%%%%%%%%

\subsection{\texorpdfstring{$k$}{k}-edge-connected graphs}

Now, we turn our attention to counting Hamiltonian cycles in the basis 
graph of the cycle matroid of $k$-edge-connected graphs for~$k \ge 3$.

%%%%%%%%%%%%%%%%%%%%%%%%%%%%%%
%%%%%%%%%% Lemma
%%%%%%%%%%%%%%%%%%%%%%%%%%%%%%
\begin{lemma}\label{lem:C4-kconn}
  If $G$ is a $k$-edge-connected graph of order $n \ge 3$ for $k \ge 3$, 
  then there are $(n-2)(k-1)$ good cycles for every edge of $\BG(M_G)$.
\end{lemma}

\begin{proof}
  Let $\Ba\Bb$ be an edge of $\BG(M_G)$, say $\Bb = \Ba -e + g$, and let $f \in \Ba-e$.
  First we show that there are $k-1$ good cycles in $\SGC(f)$. 
  
  As before, let~$X$ be the vertex set of the component of $\Ba-e$ that contains no end of~$f$, 
  let $Z$ be the vertex set of the component of $\Ba-f$ that contains no end of~$e$, and 
  let $Y = V(G) \setminus (X \cup Z)$.
  
  Suppose that $f$ is not in $C(g,\Ba)$. 
  Since $G$ is $k$-edge-connected, the edge set ${E[X\cup Y, Z]}$ contains $k-1$ edges distinct from $f$.
  It follows from Fact~\ref{fact:XY-Z} that there is a good cycle in $\SGC(f)$ for each of these edges.
  Besides, for different edges in $E[X\cup Y,Z]$, the corresponding bases $\Bc$ are different and so are their corresponding good cycles in~$\SGC(f)$.
  
  Now, suppose that $f$ is in $C(g,\Ba)$.
  It follows from Fact~\ref{fact:YZ} that there are two good cycles in $\SGC(f)$ for every edge in~$E[Y,Z]\setminus\{f\}$,  
  from Fact~\ref{fact:XY} that there is one good cycle in~$\SGC(f)$ for every edge in $E[X,Y]\setminus\{e\}$, 
  and from Fact~\ref{fact:XZ} that there is one good cycle in~$\SGC(f)$ for every edge in $E[X,Z]\setminus\{g\}$.
  Thus, as these good cycles are distinct, if~$(e(X,Y)-1) + (e(X,Z)-1) + 2(e(Y,Z)-1) \ge k-1$, we would indeed have $k-1$ good cycles in $\SGC(f)$. 
  
  By the $k$-edge-connectivity of $G$, we get that 
  \begin{eqnarray}
    |E[Y,X \cup Z]\setminus\{e,f\}| &= e(X,Y) + e(Y,Z) -2\; &\ge k-2,\label{eq:Y,XZ}\\ 
    %	|E[X,Y \cup Z]\setminus\{e,g\}| &= e(X,Y) + e(X,Z) -2\; &\ge k-2,\label{eq:X-YZ}\\
    |E[Z,X \cup Y]\setminus\{f,g\}| &= e(X,Z) + e(Y,Z) -2\; &\ge k-2.\label{eq:Z-XY}
  \end{eqnarray}
  Hence, summing~\eqref{eq:Y,XZ} and~\eqref{eq:Z-XY}, we get that $e(X,Y) + e(X,Z) + 2e(Y,Z) - 4 \ge 2k-4$, 
  and $2k-4 \ge k-1$ as $k \ge 3$.  So we have $k-1$ good cycles in $\SGC(f)$. 
  
  By Remark~\ref{rem:disjointness} and as there are $n-2$ choices for $f$, 
  there are $(n-2)(k-1)$ good cycles for every edge of $\BG(M_G)$. 
\end{proof}

%%%%%%%%%%%%%%%%%%%%%%%%%%%%%%
%%%%%%%%%% Lemma BG>BGne k-connected
%%%%%%%%%%%%%%%%%%%%%%%%%%%%%%

\begin{lemma} \label{lem:BG>BGne}
  Let $G$ be a $3$-edge-connected graph of order $n \ge 3$ and let $e$
  be an edge of~$G$.  Then $\HCs(\MG) \ge \HCs(M_{G / e})$ and
  $\HCs(\MG) \ge \HCs(M_{G \setminus e})$.
\end{lemma}

\begin{proof}
  Let $X = \{B \in \bases(\MG) \colon\, e \in B\}$ and $Y = \{B \in \bases(\MG) \colon\, e \notin B\}$.
  Note that $(X,Y)$ is a bipartition of the vertices (bases) of $\BG(M_G)$.
  Let~$G' = G/e$ and~$G'' = G\setminus e$.
  As $G'$ is $3$-edge-connected of order $n - 1 \ge 2$, the basis graph $\BGe$ has at least three vertices and is edge Hamiltonian.
  Similarly, $\BGne$ also has at least three vertices and thus is edge Hamiltonian.
  There is a one-to-one correspondence between the bases in $X$ and the bases of $\BGe$ and between the bases in $Y$ and the bases of $\BGne$.
  
  The edge set of $\BG(M_G)$ can be partitioned into:
  \begin{enumerate*}[label={\it(\roman*)}]
  \item edges with both ends in~$\BGe$, called {\em yellow edges},
  \item edges with one end in~$\BGe$ and the other one in $\BGne$, called {\em pink edges}, and
  \item edges with both ends in $\BGne$, called {\em orange edges}.
  \end{enumerate*}
  
  \setcounter{case}{0}
  \begin{case}
    Hamiltonian cycles passing through a pink edge.
  \end{case}
  
  Let $\Ba\Bb$ be a pink edge.
  For every good cycle $\GC=\Ba\Bb\Bc\Bd$, the edge $\Ba\Bd$ is an edge of $\BGe$ and $\Bb\Bc$ is an edge of $\BGne$.
  The symmetric difference of a Hamiltonian cycle of $\BGe$ containing the edge $\Ba\Bd$, the good cycle $\GC$, and
  a Hamiltonian cycle of $\BGne$ containing the edge $\Bb\Bc$ is a Hamiltonian cycle of~$\BGG$ containing the edge $\Ba\Bb$. 
  By Lemma~\ref{lem:C4-kconn}, there exists at least one such good cycle $\GC$. 
  Thus $\HC_{\Ba\Bb}(\MG) \ge \HCs(\MGe)$ and   $\HC_{\Ba\Bb}(\MG) \ge \HCs(\MGne)$.
  
  \begin{case}
    Hamiltonian cycles passing through a yellow edge.
  \end{case}
  
  First we prove that every yellow edge belongs to a good cycle $\GC$.
  Let $\Ba\Bd$ be a yellow edge, say $\Bd = \Ba - f + w$.
  The subgraph $\Bd - e$ has exactly two components.
  Since $G$ is $3$-edge-connected, there exists an edge $g$ other than $e$ and $f$ connecting the two components of $\Bd - e$.
  Therefore $e \in C(g, \Bd)$ and $\Bc = \Bd - e + g$ is a basis.
  If~$\Ba -e +g$ is a basis, we set $\Bb = \Ba - e + g$ and we have a good cycle $\GC = \Ba\Bb\Bc\Bd$.
  So, we may assume that $\Ba - e + g$ is not a basis.  
  Thus $e \not\in C(g, \Ba)$.
  Since $e \in C(g, \Bd)$, we have that $C(g,\Ba) \neq C(g,\Bd)$.
  This implies that $f \in C(g, \Ba)$ and $w \in C(g, \Bd)$ since $\Ba$ and $\Bd$ only differ by $f$ and $w$.
  Hence, we know that~$C(g, \Ba) \Delta C(g, \Bd)$, the symmetric difference of $C(g, \Ba)$ and $C(g, \Bd)$,
  \begin{enumerate*}[label=({\em \roman*})]
  \item contains $e,f,w$ but not $g$;
  \item is contained in~$\Ba \cup \Bd = \Ba + w$; and
  \item contains a cycle.
  \end{enumerate*}
  Therefore, the only cycle contained in~$C(g, \Ba) \Delta C(g, \Bd)$ is $C(w, \Ba)$.
  In this case, $\Bb = \Ba - e + w$ is a basis and there is also a good cycle $\GC = \Ba\Bb\Bc\Bd$.
  
  Now, every yellow edge $\Ba\Bd$ is in a Hamiltonian cycle $C$ of~$\BGe$. 
  Let us {\em extend} $C$ to a Hamiltonian cycle of $\BGM$ as follows. 
  Consider an edge $\Ba'\Bd'$ in $C$ other than~$\Ba\Bd$.
  The symmetric difference of $C$, a good cycle containing the yellow edge $\Ba'\Bd'$, say $\Ba'\Bb'\Bc'\Bd'$, 
  and a Hamiltonian cycle of $\BGne$ containing the edge $\Bb'\Bc'$ is an extension of $C$ to 
  a Hamiltonian cycle of $\BGG$ containing the edge~$\Ba\Bd$.
  Thus $\HC_{\Ba\Bd}(\MG) \ge \HCs(\MGne)$ because~$\Bb'\Bc'$ is in~$\HCs(\MGne)$ Hamiltonian cycles of~$\BGne$.
  
  On the other hand, since $\BGe$ is a simple graph, two Hamiltonian cycles of~$\BGe$ passing through the edge $\Ba\Bd$, 
  say $C$ and $C'$, differ in at least two edges.
  Therefore $C$ and $C'$ are extended to different Hamiltonian cycles of~$\BGG$.
  Hence $\HC_{\Ba\Bd}(\MG) \ge \HCs(\MGe)$. % 
  
  \begin{case}
    Hamiltonian cycles passing through an orange edge.
  \end{case}
  
  First we prove that every Hamiltonian cycle $C$ in $\BGne$ contains two edges, say~$\Bb\Bc$ and $\Bb'\Bc'$, each of which is in a good cycle in~$\SGC$.
  As $G$ has at least three vertices and is $3$-edge-connected, there exist an edge $f$ in~$G$ not parallel to $e$ and a basis of $\MGne$ not containing $f$.
  By traversing $C$, we pass through edges $\Bb\Bc$ and $\Bc'\Bb'$ such that ${\Bc = \Bb - f + w}$ and $\Bb' = \Bc' + f - w'$ for some elements $w$ and $w'$.
  We shall prove that there exists a $\GC$ containing $\Bb\Bc$ and a $\GC'$ containing $\Bb'\Bc'$. 
  
  As $f$ is not parallel to $e$, there exists an edge $g \in C(e,\Bb)$ other than $f$.
  Therefore $\Ba = \Bb - g + e$ is a basis.
  If $\Bc - g + e$ is also a basis, we set $\Bd = \Bc - g + e$ and obtain a good cycle~$\GC$ with $\Bb\Bc$.
  So, we may assume that $\Bc - g + e$ is not a basis.
  Thus $g \notin C(e, \Bc)$.
  Since $g \in C(e, \Bb)$, we have that $C(e, \Bb) \neq C(e, \Bc)$.
  This implies that $f \in C(e, \Bb)$ and $w \in C(e, \Bc)$ since $\Bb$ and $\Bc$ only differ by $f$ and $w$.
  In this case $\Bd = \Bc - w + e$ is a basis and we obtain a good cycle~$\GC$ with $\Bb\Bc$.
  This completes the proof for $\Bb\Bc$. The proof for $\Bb'\Bc'$ is analogous.
  
  Every orange edge $\Bb\Bc$ is in a Hamiltonian cycle $C$ of $\BGne$.
  Since $C$ contains two edges, each in a good cycle, there exists an edge $\Bb'\Bc'$, distinct from $\Bb\Bc$, in a good cycle $\GC'$, 
  say $\GC' = \Ba'\Bb'\Bc'\Bd'$. 
  The symmetric difference of $C$, the good cycle $\GC'$, and a Hamiltonian cycle of $\BGe$ passing through the edge $\Ba'\Bd'$ is 
  a Hamiltonian cycle of $\BGG$ containing the edge $\Bb\Bc$.
  Because there are $\HCs(\MGe)$ Hamiltonian cycles passing through the edge $\Ba'\Bd'$, we have that $\HC_{\Bb\Bc}(\MG) \ge \HCs(\MGe)$.
  As~$\Bb\Bc$ is in~$\HCs(\MGne)$ Hamiltonian cycles of $\BGne$, and two distinct Hamiltonian cycles differ in at least two edges, 
  we have that $\HC_{\Bb\Bc}(\MG) \ge \HCs(\MGne)$. 
\end{proof}

In order to give a bound on $\HCs(\MG)$, we define the function
$$\hc{n,k} =\min\{\HCs(\MG) \colon\, \text{$G$ is a $k$-edge-connected graph of order $n$}\}.$$

%%%%%%%%%%%%%%%%%%%%%%%%%%%%%%
%%%%%%%%%% Proop f(n,k)
%%%%%%%%%%%%%%%%%%%%%%%%%%%%%%
\begin{proposition}\label{prop:recurrence}
  For $k$, $n \ge 3$,  $\hc{n,k} \ge (n-2)(k-1)\hc{n-1,k}\hc{n,k-1}$.
\end{proposition}

\begin{proof}
Let $G$ be a $k$-edge-connected graph of order $n$ such that $\HCs(\MG) = \hc{n,k}$. 
By Lemma~\ref{lem:C4-kconn}, there are $(n-2)(k-1)$ good cycles for every edge of~$\BGG$.
Let $\Ba\Bb$ be an edge of~$\BGG$, say $\Bb = \Ba - e + g$, and let $G' = G / e$ and $G'' = G \setminus e$.
The symmetric difference of a good cycle $\GC=\Ba\Bb\Bc\Bd$, a Hamiltonian cycle of $\BGe$ containing $\Ba\Bd$, 
and a Hamiltonian cycle of $\BGne$ containing $\Bb\Bc$ is a Hamiltonian cycle of $\BGG$ containing $\Ba\Bb$.
Hence, $\Ba\Bb$ is in ${(n-2)(k-1)\HCs(\MGe)\HCs(\MGne)}$ Hamiltonian cycles of~$\BGG$.
Now, as $G'$ is {$k$-edge}-connected of order $n-1$, we have that $\HCs(\MGe) \ge \hc{n-1,k}$
and, as $G''$ is $(k-1)$-edge-connected of order $n$, we have that $\HCs(\MGne) \ge \hc{n,k-1}$.
Therefore we conclude that $\hc{n,k} = \HCs(\MG) \ge (n-2)(k-1)\hc{n-1,k}\hc{n,k-1}$.
\end{proof}

The \defi{superfactorial} $\sfac{x}$ of a positive integer $x$ is the number $x!(x-1)!\cdots0!$ 

%%%%%%%%%%%%%%%%%%%%%%%%%%%%%%
%%%%%%%%%% Theorem hc(3,k)
%%%%%%%%%%%%%%%%%%%%%%%%%%%%%%
\begin{theorem}\label{thm:hc(3,k)}
For $k \ge 3$,  $\hc{3,k} \ge \sfac{k-1}$.
\end{theorem}

\begin{proof}
We use induction on $k$. 
Let $G$ be a $k$-edge-connected graph of order three such that $\HCs(\MG)=\hc{3,k}$. 
Let $\Ba\Bb$ be an edge of $\BGG$, say $\Bb = \Ba - e + g$.
Let $G' = G/e$ and $G'' = G \setminus e$.
The graph $G'$ is $k$-edge-connected of order two. 

If $k = 3$, then $\BGe$ has three vertices and is edge Hamiltonian.
The graph $G''$ has three vertices and is $2$-edge-connected, therefore $\BGne$ has three vertices and is edge Hamiltonian.
By Lemma~\ref{lem:C4-kconn}, the edge $\Ba\Bb$ is in two good cycles in~$\SGC$.
The symmetric difference of a good cycle $\GC=\Ba\Bb\Bc\Bd$, a Hamiltonian cycle of $\BGe$ containing $\Ba\Bd$,  
and a Hamiltonian cycle of $\BGne$ containing $\Bb\Bc$ is a Hamiltonian cycle of $\BGG$ containing $\Ba\Bb$.
Hence, every edge of $\BGG$ is in two Hamiltonian cycles and $\hc{3,k} = \HCs(\MG) \ge 2$.

Suppose $k > 3$. By the induction hypothesis, $\hc{3,k-1} \ge \sfac{k-2}$.
The basis graph $\BGe$ is a complete graph on at least $k$ vertices, thus $\HCs(\MGe) \ge (k-2)!$ 
Since~$G''$ is a $(k-1)$-edge-connected graph of order three,
every edge of~$\BGne$ is in $\hc{3,k-1}$ Hamiltonian cycles.
By Lemma~\ref{lem:C4-kconn}, the edge $\Ba\Bb$ is in $k-1$ good cycles in $\SGC$.
The symmetric difference of a good cycle $\GC=\Ba\Bb\Bc\Bd$, a Hamiltonian cycle of $\BGe$ containing $\Ba\Bb$,  
and a Hamiltonian cycle of $\BGne$ containing $\Bb\Bc$ is a Hamiltonian cycle of $\BGG$ containing $\Ba\Bb$.
Hence, every edge of $\BGG$ is in $(k-2)!(k-1)\hc{3,k-1} \ge (k-1)!\sfac{k-2} = \sfac{k-1}$ Hamiltonian cycles.
Thus, $\HCs(\MG) = \hc{3,k} \geq \sfac{k-1}$.
\end{proof}

%%%%%%%%%%%%%%%%%%%%%%%%%%%%%%
%%%%%%%%%% Theorem hc(n,3)
%%%%%%%%%%%%%%%%%%%%%%%%%%%%%%
\begin{theorem}\label{thm:hc(n,3)}
For $n \ge 3$,  $\hc{n,3} \ge (n-2)!\,2^{\binom{n-1}{2}}$.
\end{theorem}

\begin{proof}
We use induction on $n$. 
Let $G$ be a $3$-edge-connected graph of order $n \ge 3$ such that $\HCs(\MG)=\hc{n,3}$. 
Let~$\Ba\Bb$ be an edge of $\BGG$, say~$\Bb = \Ba - e + g$.
Let $G' = G/e$ and~$G'' = G \setminus e$.
Note that~$G'$ is a $3$-edge-connected graph of order $n-1$. 

If $n = 3$, then $G'$ is a $3$-edge-connected graph of order two. 
So~$\BGe$ is the complete graph~$K_m$, where $m \ge 3$ is the number of edges of $G'$. 
(Remember that we remove loops of $G'$ if any.)
The graph~$G''$ has three vertices and is $2$-edge-connected, therefore $\BGne$ has at least three vertices and is edge Hamiltonian.
By Lemma~\ref{lem:C4-kconn}, the edge $\Ba\Bb$ is in two good cycles in $\SGC$.
The symmetric difference of a good cycle $\GC=\Ba\Bb\Bc\Bd$, a Hamiltonian cycle of $\BGe$ containing $\Ba\Bd$, 
and a Hamiltonian cycle of $\BGne$ containing~$\Bb\Bc$ is a Hamiltonian cycle of $\BGG$ containing $\Ba\Bb$.
Hence, every edge of $\BGG$ is in two Hamiltonian cycles, and so $\HCs(\MG) = \hc{3,3} \geq 2$.

Suppose $n > 3$.  By the induction hypothesis, $\hc{n-1,3} \ge (n-3)!\,2^{\binom{n-2}{2}}$.
Since~$G'$ is a $3$-edge-connected graph of order $n-1$, 
every edge of~$\BGe$ is in $\hc{n-1,3}$ Hamiltonian cycles.
By Theorem~\ref{thm:2-conn}, as~$G''$ is $2$-edge-connected of order $n$, 
every edge of~$\BGne$ is in $2^{n-3}$ Hamiltonian cycles.
By Lemma~\ref{lem:C4-kconn}, the edge $\Ba\Bb$ is in ${2(n-2)}$ good cycles in $\SGC$.
The symmetric difference of a good cycle~$\GC=\Ba\Bb\Bc\Bd$, a Hamiltonian cycle of~$\BGe$ containing $\Ba\Bd$, 
and a Hamiltonian cycle of $\BGne$ containing $\Bb\Bc$ is a Hamiltonian cycle of~$\BGG$ containing~$\Ba\Bb$. 
Therefore, 
every edge of~$\BGG$ is in $(n-3)!\,2^{\binom{n-2}{2}}2(n-2)2^{n-3} = (n-2)!\,2^{\binom{n-1}{2}}$ Hamiltonian cycles.
Thus, $\HCs(\MG) = \hc{n,3} \geq (n-2)!\,2^{\binom{n-1}{2}}$.
\end{proof}

%%%%%%%%%% Theorem hc(3,k): For $k \ge 3$,  $\hc{3,k} \ge \sfac{k-1}$.
%%%%%%%%%% Theorem hc(n,3): For $n \ge 3$,  $\hc{n,3} \ge (n-2)!\,2^{\binom{n-1}{2}}$.

The next theorem gives a bound on $\hc{n,k}$ for $n \ge 4$ and $k \ge 4$.  

%%%%%%%%%%%%%%%%%%%%%%%%%%%%%%
%%%%%%%%%% Theorem
%%%%%%%%%%%%%%%%%%%%%%%%%%%%%%
\begin{theorem}\label{thm:hc(n,k)} 
For $n$, $k \ge 4$, 
$$ \hc{n,k} \ \ge \ \frac{2^{\binom{n+k-4}{n-3}} \cdot 3^{\binom{n+k-7}{k-3}}}{(n-1)k}\,
                    \prod\limits^{k}_{r = 4} \bigl(r\sfac{r-1}\bigr)^{\binom{n+k-4-r}{n-4}} \cdot \prod\limits^{n}_{s = 4} (s-1)!^{\binom{n+k-4-s}{k-4}}. $$
\end{theorem}

\begin{proof}
The proof is by induction on $n+k$ and uses repeatedly Proposition~\ref{prop:recurrence}. 

For $n=k=4$, we apply Theorem~\ref{thm:hc(3,k)} to $\hc{3,4}$ and Theorem~\ref{thm:hc(n,3)} to $\hc{4,3}$:
\begin{eqnarray*}
  \hc{4,4} & \ \ge \ & 2 \cdot 3\,\hc{3,4}\,\hc{4,3} \ \ge \ 6\,\bigl[\sfac{3}\bigr] \cdot \bigl[2!\,2^{\binom{3}{2}}\bigr]
  \ = \ 3! \cdot \sfac{3}\,2 \cdot 2^3 \\
  & \  =  \ & 2^{\binom{4}{1}} \cdot \frac{3 \cdot 4}{3 \cdot 4} \, \sfac{3} \cdot 3!
  \ = \ \frac{2^{\binom{4}{1}} \cdot 3^{\binom{1}{1}}}{3 \cdot 4} \bigl(4\sfac{3}\bigr)^{\binom{0}{0}} \cdot 3!^{\binom{0}{0}}.
\end{eqnarray*}

The bound on $\hc{4,k}$ for $k \ge 5$ comes from applying Theorem~\ref{thm:hc(3,k)} to $\hc{3,k}$ and the induction hypothesis on $\hc{4,k-1}$:
\begin{eqnarray*}
\hc{4,k} & \geq &  2(k-1)\,\bigl[\sfac{k-1}\bigr] \cdot \Biggl[ \frac{2^{\binom{k-1}{1}} \cdot 3^{\binom{k-4}{k-4}}}{3(k-1)}\cdot\biggl(\,\prod\limits^{k-1}_{r = 4} \bigl(r\sfac{r-1}\bigr)^{\binom{k-1-r}{0}}\biggr) \cdot 3!^{\binom{k-5}{k-5}} \Biggr] \\
         &   =  & \frac23\,\sfac{k-1} \cdot 2^{k-1} \cdot 3^{\binom{k-3}{k-3}} \cdot \biggl(\,\prod\limits^{k-1}_{r = 4} \bigl(r\sfac{r-1}\bigr)^{\binom{k-r}{0}}\biggr) \cdot 3!^{\binom{k-4}{k-4}} \\
         &   =  & \frac{2^{\binom{k}{1}} \cdot 3^{\binom{k-3}{k-3}}}{3k} \cdot \biggl(\,\prod\limits^{k}_{r = 4} \bigl(r\sfac{r-1}\bigr)^{\binom{k-r}{0}}\biggr) \cdot 3!^{\binom{k-4}{k-4}}.
\end{eqnarray*}

Similarly, the bound on $\hc{n,4}$ for $n \ge 5$ comes from applying the induction hypothesis on $\hc{n-1,4}$ and Theorem~\ref{thm:hc(n,3)} to $\hc{n,3}$:
\begin{eqnarray*}
\hc{n,4} & \ge & (n-2)3\, \Biggl[\frac{2^{\binom{n-1}{n-4}} \cdot 3^{\binom{n-4}{1}}}{(n-2)4} \cdot \bigl(4\sfac{3}\bigr)^{\binom{n-5}{n-5}} \cdot \prod\limits^{n-1}_{s = 4} (s-1)!^{\binom{n-1-s}{0}}\Biggr] \cdot \Bigl[(n-2)!\,2^{\binom{n-1}{2}}\Bigr] \\
         &  =  & \frac34\,2^{\binom{n-1}{n-4}} \cdot 3^{n-4} \cdot \bigl(4\sfac{3}\bigr)^{\binom{n-4}{n-4}} \cdot \biggl(\,\prod\limits^{n-1}_{s = 4} (s-1)!^{\binom{n-s}{0}}\biggr) \cdot (n-2)!\,2^{\binom{n-1}{n-3}} \\
         &  =  & \frac{2^{\binom{n-1}{n-4}+\binom{n-1}{n-3}} \cdot 3^{n-3}}{(n-1)4}\,(n-1) \cdot \bigl(4\sfac{3}\bigr)^{\binom{n-4}{n-4}} \cdot \biggl(\,\prod\limits^{n-1}_{s = 4} (s-1)!^{\binom{n-s}{0}}\biggr) \cdot (n-2)!\\
         &  =  & \frac{2^{\binom{n}{n-3}} \cdot 3^{\binom{n-3}{1}}}{(n-1)4} \cdot \bigl(4\sfac{3}\bigr)^{\binom{n-4}{n-4}} \cdot \biggl(\,\prod\limits^{n}_{s = 4} (s-1)!^{\binom{n-s}{0}}\biggr). 
\end{eqnarray*}

Finally, the bound on $\hc{n,k}$ for $n$, $k \ge 5$ comes from applying the induction hypothesis on both~$\hc{n-1,k}$ and $\hc{n,k-1}$:
\begin{eqnarray*}
\hc{n,k} & \geq & (n{-}2)(k{-}1)\frac{2^{\binom{n+k-5}{n-4}} \cdot 3^{\binom{n+k-8}{k-3}}}{(n{-}2)k}\,\prod\limits^{k}_{r = 4} \bigl(r\sfac{r-1}\bigr)^{\binom{n+k-5-r}{n-5}} \cdot \prod\limits^{n-1}_{s = 4} (s-1)!^{\binom{n+k-5-s}{k-4}} \\
         &      & \phantom{(n{-}2)(k{-}\!}\ \cdot \frac{2^{\binom{n+k-5}{n-3}} \cdot 3^{\binom{n+k-8}{k-4}}}{(n-1)(k-1)}\,\prod\limits^{k-1}_{r = 4} \bigl(r\sfac{r-1}\bigr)^{\binom{n+k-5-r}{n-4}} \cdot \prod\limits^{n}_{s = 4} (s-1)!^{\binom{n+k-5-s}{k-5}} \\
         &   =  & \frac{2^{\binom{n+k-4}{n-3}} \cdot 3^{\binom{n+k-7}{k-3}}}{(n-1)k}\, \biggl(\,\prod\limits^{k-1}_{r = 4} \bigl(r\sfac{r-1}\bigr)^{\binom{n+k-4-r}{n-4}}\biggr) \cdot \bigl(k\sfac{k-1}\bigr)^{\binom{n-5}{n-5}} \\
         &      & \phantom{\frac{2^{\binom{n+k-4}{n-3}} 3^{\binom{n+k-7}{k-3}}}{(n-1)k}\,} \cdot \biggl(\,\prod\limits^{n-1}_{s = 4} (s-1)!^{\binom{n+k-4-s}{k-4}}\biggr) \cdot (n-1)!^{\binom{k-5}{k-5}} \\
         &   =  & \frac{2^{\binom{n+k-4}{n-3}} \cdot 3^{\binom{n+k-7}{k-3}}}{(n-1)k}\,\prod\limits^{k}_{r = 4} \bigl(r\sfac{r-1}\bigr)^{\binom{n+k-4-r}{n-4}} \cdot \prod\limits^{n}_{s = 4} (s-1)!^{\binom{n+k-4-s}{k-4}}.
\end{eqnarray*}
This completes the proof of the theorem.
\end{proof}

The following corollary follows from mathematical manipulations on the right side of the inequality given by
Theorem~\ref{thm:hc(n,k)} and it gives a more explicit and concise expression.

\begin{corollary}
For $n > k \ge 5$, 
$$\hc{n,k} > \prod\limits^{n}_{r=3} \sfac{r-1}^{\binom{n+k-5-r}{n-6}+\binom{n+k-4-r}{n-4}+\binom{n+k-5-r}{k-5}}.$$
\end{corollary}

\begin{proof}
We start proving two auxiliary equalities that shall be used to prove the corollary.
Firstly, 
\begin{align}
2^{\binom{n+k-4-2}{n-4}} \cdot 3^{\binom{n+k-4-3}{k-3}} \cdot \prod\limits^{k}_{r = 4} r^{\binom{n+k-4-r}{n-4}} 
        & =  \prod\limits^{k}_{r = 2}       r ^{\binom{n+k-4-r}{n-4}} \nonumber \\
        & =  \prod\limits^{k}_{r = 2}       r!^{\binom{n+k-4-r}{n-4}-\binom{n+k-4-(r+1)}{n-4}} \label{n_geq_5} \\
        & =  \prod\limits^{k}_{r = 2}       r!^{\binom{n+k-5-r}{n-5}} \nonumber \\
        & =  \prod\limits^{k}_{r = 2}   \sfac{r}^{\binom{n+k-5-r}{n-5}-\binom{n+k-5-(r+1)}{n-5}} \label{n_geq_6} \\
        & =  \prod\limits^{k}_{r = 2}   \sfac{r}^{\binom{n+k-6-r}{n-6}} \nonumber \\
        & =  \prod\limits^{k+1}_{r = 3} \sfac{r-1}^{\binom{n+k-6-(r-1)}{n-6}} \nonumber \\
        & =  \prod\limits^{n}_{r = 3}   \sfac{r-1}^{\binom{n+k-5-r}{n-6}}. \nonumber
\end{align}
Equalities \eqref{n_geq_5} and \eqref{n_geq_6} follow from the hypothesis that $n \geq 6$.

Secondly, 
\begin{align}
  2^{\binom{n+k-4-3}{k-4}} \cdot
        \prod\limits^{n}_{s = 4}      (s-1)!^{\binom{n+k-4-s}{k-4}} \nonumber 
    & = \prod\limits^{n}_{s = 3}      (s-1)!^{\binom{n+k-4-s}{k-4}} \nonumber \\
    & = \prod\limits^{n}_{s = 3} \sfac{s-1}^{\binom{n+k-4-s}{k-4}-\binom{n+k-4-(s+1)}{k-4}} \label{k_geq_5} \\
    & = \prod\limits^{n}_{s = 3} \sfac{s-1}^{\binom{n+k-5-s}{k-5}}. \nonumber 
\end{align}
Equality \eqref{k_geq_5} follows from the hypothesis that $k \geq 5$.

Thus, by Theorem~\ref{thm:hc(n,k)}, we have that
\begin{align}
\hc{n,k} & \ge  \frac{2^{\binom{n+k-4}{n-3}} \cdot 3^{\binom{n+k-7}{k-3}}}{(n-1)k}
                    \cdot \prod\limits^{k}_{r = 4} \bigl(r\sfac{r-1}\bigr)^{\binom{n+k-4-r}{n-4}} 
                    \cdot \prod\limits^{n}_{s = 4} (s-1)!^{\binom{n+k-4-s}{k-4}} \nonumber \\
         &  =   \frac{2^{\binom{n+k-4}{n-3}} \cdot 3^{\binom{n+k-7}{k-3}}}{(n-1)k}
                   \cdot \prod\limits^{k}_{r = 4} r^{\binom{n+k-4-r}{n-4}} \nonumber \\
         &    \phantom{=  \frac{2^{\binom{n+k-4}{n-3}} \cdot 3^{\binom{n+k-7}{k-3}}}{(n-1)k} \; }
                    \cdot \prod\limits^{k}_{r = 4} \sfac{r-1}^{\binom{n+k-4-r}{n-4}} \nonumber \\
         &    \phantom{=  \frac{2^{\binom{n+k-4}{n-3}} \cdot 3^{\binom{n+k-7}{k-3}}}{(n-1)k} \; }
                    \cdot \prod\limits^{n}_{s = 4} (s-1)!^{\binom{n+k-4-s}{k-4}} \nonumber \\
         &  = \frac{2^{\binom{n+k-5}{n-4}}}{(n-1)k}
                            \cdot \biggl(2^{\binom{n+k-6}{n-4}} \cdot 3^{\binom{n+k-7}{k-3}}\prod\limits^{k}_{r = 4} r^{\binom{n+k-4-r}{n-4}} \biggr) \label{eq_binomial}\\
         &    \phantom{= \frac{2^{\binom{n+k-5}{n-4}}}{(n-1)k}} \;
                           \cdot \biggl(2^{\binom{n+k-7}{n-4}}\prod\limits^{k}_{r = 4} \sfac{r-1}^{\binom{n+k-4-r}{n-4}}\biggr)\nonumber \\
         &    \phantom{= \frac{2^{\binom{n+k-5}{n-4}}}{(n-1)k}} \;
                            \cdot \biggl(2^{\binom{n+k-7}{n-3}} \cdot \prod\limits^{n}_{s = 4} (s-1)!^{\binom{n+k-4-s}{k-4}}\biggr) \nonumber \\
         &  =   \frac{2^{\binom{n+k-5}{n-4}}}{(n-1)k} \cdot 
                        \biggl(\,\prod\limits^{n}_{r = 3} \sfac{r-1}^{\binom{n+k-5-r}{n-6}}\biggr) \label{eq_previous} \\
         &    \phantom{= \frac{2^{\binom{n+k-5}{n-4}}}{(n-1)k}} \;            
                        \cdot \biggl(\,\prod\limits^{n}_{r = 3} \sfac{r-1}^{\binom{n+k-4-r}{n-4}}\biggr) \nonumber \\
         &    \phantom{= \frac{2^{\binom{n+k-5}{n-4}}}{(n-1)k}} \;
                        \cdot \biggl(\,\prod\limits^{n}_{r = 3} \sfac{r-1}^{\binom{n+k-5-r}{k-5}}\biggr) \nonumber \\ 
         &  =   \frac{2^{\binom{n+k-5}{n-4}}}{(n-1)k} \cdot \prod\limits^{n}_{r = 3} \sfac{r-1}^{\binom{n+k-5-r}{n-6}+\binom{n+k-4-r}{n-4}+\binom{n+k-5-r}{k-5}} \nonumber \\
         &  >   \prod\limits^{n}_{r = 3} \sfac{r-1}^{\binom{n+k-5-r}{n-6}+\binom{n+k-4-r}{n-4}+\binom{n+k-5-r}{k-5}}. \nonumber 
%         &  >   \prod\limits^{n}_{r = 3} \sfac{r-1}^{\binom{n+k-5-r}{n-5}+\binom{n+k-5-r}{k-5}} \nonumber \\
%         & \ge  \prod\limits^{n}_{r = 5} \sfac{r-1}^{2\binom{n+k-5-r}{n-r}}. \nonumber 
\end{align}%
Equality~\eqref{eq_binomial} holds because 
{\scriptsize $\binom{n+k-4}{n-3} = \binom{n+k-7}{n-3}+\binom{n+k-7}{n-4} + \binom{n+k-6}{n-4} + \binom{n+k-5}{n-4}$,}
and~\eqref{eq_previous} follows from the two previous equalities. 
\end{proof}

%\newpage

%%%%%%%%%%%%%%%%%%%%%%%%%%%%%%%%%%%%%%%%%%%%%%%%%%%%%%%%%%%%%%%%%%%%%%%%
%%%%%%%%%   Lattice Path 
%%%%%%%%%%%%%%%%%%%%%%%%%%%%%%%%%%%%%%%%%%%%%%%%%%%%%%%%%%%%%%%%%%%%%%%%

\section{Generalized Catalan matroids}\label{sec:latticepathmatroids}

In this section we address a special class of transversal matroids introduced by Bonin, de Mier, and Noy~\cite{BMN03}. We follow the description of Bonin and de Mier~\cite{BoMi06} and Stanley~\cite{St12}. 

Let $S$ be a subset of $\mathbb{Z}^d$.
A \defi{lattice path} $L$ in $\mathbb{Z}^d$ of length $k$ with steps in $S$ is a sequence $v_0, \ldots, v_k \in \mathbb{Z}^d$ such that each consecutive difference $s_j = v_j - v_{j-1}$ lies in~$S$.
We call $s_j$ the $j$th step of the lattice path $L$.
We say that $L$ starts at $v_0$ and ends at $v_k$, or simply that $L$ goes from $v_0$ to $v_k$.

All lattice paths we consider are in $\mathbb{Z}^2$, start
at~$(0,0)$ and end at $(m,r)$, and use steps in $S = \{(1,0),
(0,1)\}$.
We call the steps $(1,0)$ and $(0,1)$ as \defi{East} ($E$) and
\defi{North}~($N$), respectively.
Sometimes it is convenient to represent a lattice path $L$ as a
sequence of steps;
that is, as a word of length $m+r$ on the alphabet~$\{E,N\}$;
other times, as a subset of $\{1, \ldots, m+r\} = [m{+}r]$, say $\{ j \colon\, \text{$j$th step of } L \text{ is } $\north$\}$.

%%%%% P is in the bottom
%%%%% Q is in the top

Let $P$ and $Q$ be lattice paths from $(0,0)$ to $(m,r)$ with $P$
never going above~$Q$.
Let~$\calP$ be the set of all lattice paths from $(0,0)$ to
$(m,r)$ that go neither below~$P$ nor above $Q$.
For each~$i$ with $1 \leq i \leq r$, let $A_i$ be the set
$$A_i = \{j \colon\, \text{$j$th step is the $i$th North for some path in } \calP\}.$$ 
Observe that $A_1, \ldots,A_r$ are intervals $A_i = [a_i, b_i]$
in $[m{+}r]$.
Moreover $a_1 < \cdots < a_r$ and $b_1 < \cdots < b_r$; and $a_i$
and $b_i$ correspond to the positions of the $i$th North step of~$Q$
and~$P$, respectively.
An example is shown in Figure~\ref{fig:latticepaths}.

\begin{figure}[h]
\centering
\includegraphics[width=0.95\linewidth]{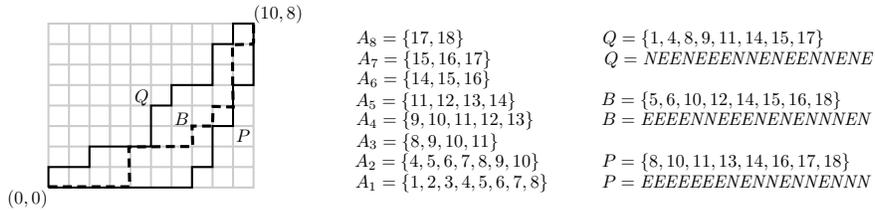}
\caption{Lattice paths $P$ and $Q$ from $(0,0)$ to $(10,8)$ and the corresponding sets $A_1,\ldots,A_8$. 
Representations of $P$ and $Q$ as words of length $10 + 8$ in the alphabet $\{E,N\}$ and as subsets of $[10+8]$.
Lattice path~$B$ goes neither below $P$ nor above $Q$ and its representations as a word and as a subset.}
\label{fig:latticepaths}
\end{figure}

Let $M[P,Q]$ be the transversal matroid on the ground set
$[m{+}r]$ and $(A_1,\ldots,A_r)$ its presentation.
We call $(A_1,\ldots,A_r)$ the \emph{standard presentation}
of~$M[P,Q]$.
Note that~$M[P,Q]$ has rank $r$ and corank (or nullity) $m$.  A
transversal matroid is a \emph{lattice path matroid} if it is
a matroid of the type $M[P,Q]$.
Each basis of $M[P,Q]$ corresponds to a lattice path from $(0,0)$
to $(m,r)$ that goes neither below~$P$ nor above $Q$.
Figure~\ref{fig:latticepaths} shows an illustration of a matroid
$M[P,Q]$ and a basis~$B$.

Let $M[P,Q]$ be a lattice path matroid.
Let $P = y_1 \cdots y_i \cdots y_{m+r}$ ($= y^{[m{+}r]}$) and ${Q
= x_1 \cdots x_i \cdots x_{m+r}}$ ($= x^{[m{+}r]}$), with $x_i,
y_i \in \{\north, \east \}$ for $i\in [m{+}r]$.
A \defi{generalized Catalan matroid} is a lattice path matroid
$M[P,Q]$, where $P = \mathit{E}^m\mathit{N}^r$.
We simply write $M[Q]$ for generalized Catalan matroids.
The class of generalized Catalan matroid is
minor-closed~\cite[Theorem~4.2]{BoMi06}.
The \defi{$k$-Catalan matroid} is the generalized Catalan
matroid $M[(\north \east)^{k}]$; that is, $Q = (\mathit{NE})^k$.

%%%%%%%%%%%%%%%%%%%%%%%%%%%%%%
%%%%%%%%%% C4-latticepath
%%%%%%%%%%%%%%%%%%%%%%%%%%%%%%

\begin{lemma}\label{lem:C4-latticepath}
Let $M[Q]$ be a generalized Catalan matroid of rank $r$ and
corank $m$, for $m \ge r \ge 2$, with neither a loop nor an isthmus.
Then, every edge of $\BG(M[Q])$ is in~${r-1}$ good cycles. 
\end{lemma}

\begin{proof}
As $M[Q]$ has neither a loop nor an isthmus, the first step of $Q$ is
North and the last one is East.
By convenience, we consider the bases of $\MQ$ as words of length $m+r$ in the
alphabet $\{\north, \east\}$.

Let $\Ba\Bb$ be an edge of $\BGMQ$, say $\Bb = \Ba - e+ g$.
Thus, 
$\Ba = x^{[m{+}r]}$, 
$\Bb = y^{[m{+}r]}$, 
and there exist indices $e$ and $g$ such that $x_e = y_g =
\north$, $x_g = y_e = \east$, and $x_{\ell} = y_{\ell}$ for $\ell
\neq e, g$.
Without loss of generality we may assume that~$e < g$.

\setcounter{case}{0}
\begin{case}
There exists an  index $\ell$ less than $e$ (and therefore less
than $g$) such that~$x_{\ell} = y_{\ell} = N$ (Figure~\ref{fig:latticeCaso1}).
\end{case}

Let $f$ be the least index such that $x_f = y_f = N$.
For every index $w$ such that $x_w = y_w =\east$, basis $\Bd$
rises by switching $x_w$ for $\north$ and $x_f$ for $\east$ in
$\Ba$ and basis $\Bc$ rises by switching $y_w$ for $\north$ and
$y_f$ for $\east$ in $\Bb$; that is, $\Bd = \Ba - f + w$ and $\Bc
= \Bb - f + w$.
Since the first step of $Q$ is North and the last one is East,
the paths corresponding to the words $\Bc$ and $\Bd$,
respectively, are in $\MQ$.
Thus, for every common $\east$ step of $\Ba$ and $\Bb$, we obtain
a good cycle~$\GC$.
Therefore, there are $m-1$ good~$\GC$ passing through the edge
$\Ba\Bb$.

\begin{figure}[h]
\centering
\includegraphics[width=0.95\linewidth]{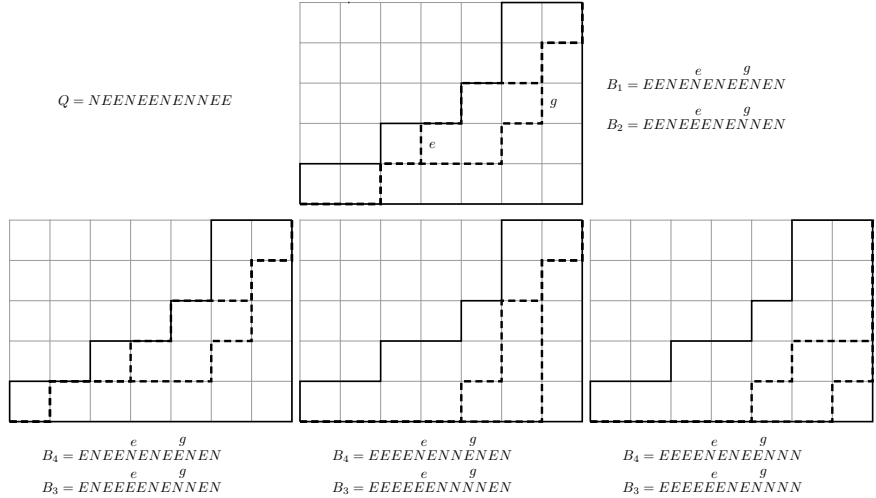}
\caption{Illustration of lattice paths corresponding to Case 1.}
\label{fig:latticeCaso1}
\end{figure}

\begin{case}
There exists an index $\ell$ greater that $g$ (and therefore
greater than $e$) such that $x_{\ell} = y_{\ell} = \east$.
\end{case}

Let $w$ be the last index such that $x_w = y_w = \east$.
For every index $f$ such that $x_f = y_f =\north$,  basis $\Bd$
rises by switching $x_f$ for $\east$ and $x_w$ for $\north$ in
$\Ba$ and basis $\Bc$ rises by switching $y_f$ for $\east$ and
$y_w$ for $\north$ in $\Bb$; that is,  $\Bd = \Ba - f + w$ and $\Bc
= \Bb - f + w$.
Since the first step of $Q$ is North and the last one is East,
the paths corresponding to the words $\Bc$ and $\Bd$,
respectively, are in $\MQ$.
Thus, for every common $\north$ step of $\Ba$ and $\Bb$, we obtain
a good cycle~$\GC$.
Therefore, there are $r-1$ good~$\GC$ passing through the edge
$\Ba\Bb$.

\begin{case}
There exist no indices $\ell$ and $\ell'$ with $\ell < e$  and
$\ell' > g$ such that~$x_{\ell} = y_{\ell} = N$ and $x_{\ell'} =
y_{\ell'} = \east$.
\end{case}

Thus, $x_e$ is the first $\north$ in $\Ba$ and $x_g$ is the last
$\east$.
Let $x_h$ be the penultimate~$\east$ in~$\Ba$.
Such~$x_h$ exists because~$m \ge r \ge 2$.
As $y_g$ is $\north$ in $\Bb$, $y_h$ is the last $\east$ in
$\Bb$.

In order to count the number of good cycles, we partition the
$\north$'s in the words corresponding to the bases
$\Ba$ and $\Bb$ in maximal
\emph{blocks}, and for each $\north$ we shall show a good cycle
associated  with it.

{\bf Block of Type I.}
Consider the block $x_i \cdots x_{w-1}x_{w}$ such that $x_i =
\cdots = x_{w-1} = \north$ and $x_{w} = \east$ with $e < i < w < g$.

Also, we have that $y_i =
\cdots = y_{w-1} = \north$ and $y_{w} = \east$.
For every $f \in \{i, \ldots, w-1\}$, basis $\Bd$
rises by switching $x_f$ for $\east$ and $x_{w}$ for $\north$ in
$\Ba$ and basis $\Bc$ rises by switching $y_f$ for $\east$ and
$y_{w}$ for $\north$ in $\Bb$; that is, $\Bd = \Ba - f + w$ and $\Bc
= \Bb - f + w$.

{\bf Block of Type II.}
Consider the block $x_i\cdots x_{w-1}x_w$ such that
$x_i = \cdots x_{w-1} = \north$ and $x_w = \east$
with $e < i < w = g$.

Also, we have that $y_i = \cdots = y_{w-1} = \north$.
Let $x_h$ the penultimate $\east$ in $\Ba$.
As $y_w = y_g$ is $\north$ in $\Bb$, $y_h$ is the last $\east$ in
$\Bb$.
For every $f \in \{i, \ldots, g-1\}$, basis $\Bd$
rises by switching $x_f$ for $\east$ and $x_g$ for $\north$ in
$\Ba$, and basis $\Bc$ rises by switching $y_f$ for $\east$ and
$y_h$ for $\north$ in $\Bb$; that is, $\Bd = \Ba - f + g$ and $\Bc
= \Bb - f + h$.

{\bf Block of Type III.}
Consider a block $x_{g+1}\cdots x_{m+r}$ of $\north$'s in $\Ba$.

Also, we have that $y_{g+1}\cdots y_{m+r}$ is a block of
$\north$'s in $\Bb$.
For every element $f~\in~\{{g+1}, \ldots, {m+r}\}$, basis $\Bd$
rises by switching $x_f$ for $\east$ and $x_g$ for $\north$ in
$\Ba$, and basis $\Bc$ rises by switching $y_f$ for $\east$ and
$y_h$ for $\north$ in $\Bb$; that is, $\Bd = \Ba - f + g$ and $\Bc
= \Bb - f + h$.

Since every $\north$ distinct of $x_e$ belongs to some type of
block, we get $r -1$ good~$\GC$ passing through the edge~$\Ba\Bb$. 
\end{proof}

Bonin and de Mier~\cite{BoMi06} observed that the class of all
generalized Catalan matroids is closed under duals.
Moreover, a basis $B^*$ of the dual of $M[P,Q]$ corresponds to
the $\east$ steps of the basis $B$ in $M[P,Q]$.
Therefore, the following is a consequence of this fact and
Lemma~\ref{lem:C4-latticepath}.

\begin{corollary}\label{cor:C4-latticepath}
  For $r, m \ge 2$, let $M[Q]$ be a generalized Catalan matroid of
  rank~$r$ and corank~$m$, with neither a loop nor an isthmus.
  Then every edge of $\BG(M[Q])$ is in $\min \{r-1, m-1\}$ good cycles.
%$\endproof$
\end{corollary}

Let $M[P,Q]$ be a lattice path matroid.
Let $P = y^{[m{+}r]}$ and $Q = x^{[m{+}r]}$, with $x_i, y_i \in
\{\north, \east \}$ for $i\in [m{+}r]$.
Assume $e$ is neither a loop nor an isthmus.
In~\cite{BoMi06} was observed that:
\begin{enumerate}[label=(\textrm{\arabic*})]
\item $M[P,Q] \setminus e$ is the lattice path matroid $M[P',Q']$
where the upper bounding path~$Q'$ is formed by deleting from $Q$
the first $\east$ step that is at or after step $e$; the lower
bounding path $P'$ is formed by deleting from $P$ the last
$\east$ step that is at or before step $x$.
%%%%%%%%%%
\item $M[P,Q] / e$ is the lattice path matroid $M[P'',Q'']$ where
the upper bounding path $Q''$ is formed by deleting from $Q$ the
last $\north$ step that is at or before step $e$; the lower
bounding path $P''$ is formed by deleting from $P$ the first
$\north$ step that is at or after step $e$.
\end{enumerate}

\begin{obs}\label{obs:Cat-minor}
  If the $k$-Catalan matroid is a minor of the generalized Catalan
  matroid $M[Q]$, then for every element $e$ of $M[Q]$,
  the ${(k-1)}$-Catalan matroid is a minor of both the
  generalized Catalan matroid $M[Q] \backslash e$ and $M[Q] \slash e$.
\end{obs}
In fact, Observation~\ref{obs:Cat-minor} also holds if we replace
generalized Catalan matroid for lattice path matroid.

For the class of generalized Catalan matroids, we define the function

$$\hcl{k} = \min\{\HCs(M[Q]) \colon\, M[Q] \text{ has a $k$-Catalan matroid as a minor}\}.$$

\begin{proposition}\label{prop:latticerecurrence}
For $k \ge 2$, $\hcl{k} \ge (k-1)\hcl{k-1}^2$.
\end{proposition}

\begin{proof}
Let $M[Q] = M_Q$ be a generalized Catalan matroid such that $\HCs(M_Q) = \hcl{k}$.
We may assume that $M_Q$ has neither a loop nor an isthmus.
Thus, both the rank and corank of $M$ are at least $k$.
By Corollary~\ref{cor:C4-latticepath}, there are $\min\{r-1,m-1\} \ge k-1$ good cycles for every edge of $\BG(M_Q)$.
Let $\Ba\Bb$ be an edge of $\BG(M_Q)$, say $\Ba = \Bb - e + g$, and let $M' = M_Q \backslash e$ and $M'' = M_Q \slash e$.
It follows from Observation~\ref{obs:Cat-minor} that both $M'$ and $M''$ contain a $(k-1)$-Catalan matroid as a minor.
Thus $\HCs(M')\ge \hcl{k-1}$ and $\HCs(M'') \ge \hcl{k-1}$.
Therefore we conclude that $\hcl{k} \ge (k-1)\hcl{k-1}^2$.
\end{proof}

%%%%%%%%%%%%%%%%%%%%
%%%%%%%%%% Theorem Generalized Catalan matroid
%%%%%%%%%%%%%%%%%%%%

% \begin{theorem}\label{thm:latticepath}
% Let $M[Q]$ be a generalized Catalan matroid of rank $r$ and corank $m$, with neither ithsmus nor loops.
% Let $k$ be the largest integer such that the $k$-Catalan matroid
% $M[(\north \east)^{k}])$ is a minor of $M[Q]$.
% Then, every edge in the basis graph $\BG(M[Q])$ is in
% $\sfac{k-1}$ Hamiltonian cycles, except possibly if~$\BG(M[Q])$
% is $K_2$.
% \end{theorem}

\begin{theorem}\label{thm:latticepath}
For $k \ge 2$, $\hcl{k} \ge \sfac{k-1}\sfac{k-2}$.
\end{theorem}

\begin{proof}
The proof is by induction on $k$.
We write simply $M_Q$ instead $M[Q]$.
Let~$M_Q$ be a generalized Catalan matroid such that $\HCs(M_Q) = \hcl{k}$.
We may assume that $M_Q$ has neither a loop nor a isthmus.
In particular, $M_Q$ has both rank and corank at least $k$.
Let $k = 2$.
So $\BG(M_Q)$ has at least three vertices and is edge Hamiltonian.
Therefore $\hcl{2} \ge 1 = \sfac{1}\sfac{0}$.

Now let $k \ge 3$. 
% So, $M_Q$ has a $k$-Catalan matroid as a minor.
Let $\Ba\Bb$ be an edge of $\BG(M_Q)$, say $\Bb = \Ba - e + g$.
By Corollary~\ref{cor:C4-latticepath}, the edge $\Ba\Bb$ is in 
$\min\{r-1,m-1\} \ge k-1$ good cycles.
Consider $M' = M_Q \backslash e$ and $M'' = M_Q \slash e$.
By Observation~\ref{obs:Cat-minor}, the $(k-1)$-Catalan matroid
is a minor of both $M'$ and $M''$.
Thus, by the induction hypothesis, $\HCs(M'), \HCs(M'') \ge \hcl{k-1} \ge \sfac{k-2}\sfac{k-3}$.
Hence, every edge of $\BG(M_Q)$ is in $(k-1)\bigl(\sfac{k-2}\sfac{k-3}\bigr)^2 \ge \sfac{k-1}\sfac{k-2}$ Hamiltonian cycles.
\end{proof}

\subsection{Uniform matroids}\label{sec:uniform}
 
Recall that the set of bases of the \emph{uniform matroid} of
rank~$r$ on~$n$
elements, denoted by $U_{r,n}$, consists of all
$r$-subsets of $[n]$.
Also, $U_{r,n}$ can be considered as the lattice path matroid $M[P,Q]$ where $Q = N^{r}E^{n-r}$ and $P = E^{n-r}N^{r}$.

Let $\Ba\Bb$ be an edge of $\BG(U_{r,n})$, say $\Ba = \Bb - e + g$.
So, we have that $\Ba = \{e, f_2, \ldots, f_r\}$ and $\Bb = \{g, f_2, \ldots, f_r\}$, with $f_i \in [n] \setminus \{e,g\}$ for $i \in \{2, \ldots, r\}$.
For every  $w$ in $[n] \setminus \{e,g,f_2, \ldots, f_r\}$,
we can obtain three types of good cycles in $\GC$ by 
replacing an $f_i$ by $w$ as shown in Table~\ref{tab:C4-uniform}.
We thus have the following result.

\begin{table}
  \caption{The three types of good cycles for $U_{r,n}$.}
  \label{tab:C4-uniform}
  \centering
  \begin{tabu} to\linewidth {XX}
    \begin{tabular}{|c|l l|c|}\hline
      $\Ba$ & $\{e,f_i,\ldots\}$ & $\{g,f_i,\ldots\}$ & $\Bb$\\
      $\Bd$ & $\{e,w,\ldots\}$ & $\{g,w,\ldots\}$ & $\Bc$\\
      \hline
    \end{tabular}
    &
    \begin{tabular}{|c|l l|c|}\hline
      $\Ba$ & $\{e,f_i,\ldots\}$ & $\{g,f_i,\ldots\}$ & $\Bb$\\
      $\Bd$ & $\{e,g,\ldots\}$ & $\{g,w,\ldots\}$ & $\Bc$\\
      \hline
    \end{tabular}
    \\[2mm]
    \multicolumn2{c}{
      \begin{tabular}{|c|l l|c|}\hline
        $\Ba$ & $\{e,f_i,\ldots\}$ & $\{g,f_i,\ldots\}$ & $\Bb$\\
        $\Bd$ & $\{e,w,\ldots\}$ & $\{w,f_i,\ldots\}$ & $\Bc$\\
        \hline
      \end{tabular}
    }
  \end{tabu}
\end{table}

\begin{proposition}\label{prop:C4-uniform}
  Let $n > r \geq 1$ be integers.
  Then every edge of $\BG(U_{r,n})$ is in $3(n-r-1)(r-1)$ good cycles.
%$\endproof$
\end{proposition}     

Finally, the next theorem can be proved by induction on the
number of elements of the matroid, applying
Proposition~\ref{prop:C4-uniform}, and following the same
strategy as above.

%%%%%%%%%%%%%%%%%%%%%%%%%%%%%%
%%%%%%%%%% Theorem
%%%%%%%%%%%%%%%%%%%%%%%%%%%%%%

\begin{theorem}
    Let $n > r \geq 1$ be integers.
    Then every edge of $\BG(U_{r,n})$ is in 
    $\bigl((n-r-1)!(r-1)!\bigr)^{\min\{n-r-1, r-1\}}$ Hamiltonian cycles.
%$\endproof$
\label{thm:uniform}
\end{theorem}

%\section{Concluding remarks}

%%%%%%%%%%%%%%%%%%%%%%%%%%%%%%%%% Bibliography

%\begin{thebibliography}
%\end{thebibliography}

\bibliographystyle{siamplain}
\bibliography{CountingHCs}

\end{document}